\documentclass[reqno]{amsart}
\usepackage{hyperref}
\usepackage{graphicx}
\usepackage[latin1]{inputenc}
\usepackage[english,activeacute]{babel}

\numberwithin{equation}{section}
\newtheorem{theorem}{Theorem}[section]
\newtheorem{lemma}[theorem]{Lemma}
\newtheorem{remark}[theorem]{Remark}
\newtheorem{definition}[theorem]{Definition}
\newtheorem{proposition}[theorem]{Proposition}

\title[Logarithmic NLS equation on star graphs]
{Logarithmic NLS equation on star graphs: existence and stability of standing waves}
\author[Alex H. Ardila]{}
\email{ardila@impa.br}

\subjclass[2010]{76B25, 35Q51, 35Q55, 35J60, 37K40, 34B37}
\keywords{Logarithmic Schr\"{o}dinger equation; star graphs;  stability}

\begin{document}
\maketitle


\centerline{\scshape Alex H. Ardila}
{\footnotesize
 \centerline{Instituto Nacional de Matem\'atica Pura e Aplicada - IMPA,}
\centerline{Estrada Dona Castorina 110, CEP 22460-320, Rio de Janeiro, RJ, Brazil.}
} 

\begin{abstract}
In this paper we consider the logarithmic Schr\"{o}dinger equation on a star graph. By using a compactness method, we construct a unique global solution of the associated Cauchy problem in a suitable functional framework.  Then we show the existence of several families of standing waves. We also prove the existence of ground states as minimizers of the action on the Nehari manifold. Finally, we show that the ground states are orbitally stable via a variational approach.\\
\end{abstract}

\section{Introduction}
Partial differential equations on graphs, or on higher-dimensional `networked' domains, 
arise naturally in many topics of physics such as optics, acoustics, condensed matter 
and polymer physics. Modern applications of PDEs on graphs include machine mechatronics, 
biology, electrical and communication networks and traffic flow. We refer to \cite{PDEOM} for further information and bibliography. 
Earlier, the linear Schr\"{o}dinger equation on a metric graph was subject of extensive research due to its applications in quantum chemistry, nanotechnologies and mesoscopic physics (see \cite{BMQT,QGAH} and references therein). Studies of the nonlinear Schr\"{o}dinger equation on graphs have started appearing recently.  In particular, existence and stability of standing waves for nonlinear Schr\"{o}dinger equation on a star graph with a power nonlinearity $\left|u\right|^{p-1}u$ have been studied extensively. Among such works, let us mention 
\cite{VCN, FSOH, OSCE, FNCN, NQTG, AQFF,ADN1, ADNP,AESM, DSED,RJJ,STW,FO,LFF, NNOOA}. 

In recent years, the logarithmic NLS equation has attracted a great deal of attention from both the mathematicians and physicists (see e.g. \cite{AHA1,CZ, SMSA}); this equation is applied in many branches of physics, e.g., quantum optics, nuclear physics, fluid dynamics, geophysics and  Bose-Einstein condensation (see, e.g. \cite{APLES} and references therein).

To set the stage, let $\Gamma$ be a star graph consisting of a central vertex $c$ and $N$ edges (half-lines) attached to it. For simplicity, each edge will be identified with the  positive semi-axis $J_{e}= (0, +\infty)$, where zero corresponds to the central vertex $c$. Thus we see that we can  identify $\Gamma$ with the (disjoint) union of the intervals $J_{e}=(0, +\infty)$, $e=1$, $\ldots$, $N$, augmented by the central vertex. The following notation will be convenient: given a function on the graph $u: \Gamma\longrightarrow \mathbb{C}^{N}$, its restriction to the semi-axis $J_{e}$  is denoted with $u_{e}$. Moreover, we will denote with $u_{e}(0)$ the limit of $u_{e}(x)$ as $x\rightarrow 0$ in $J_{e}$.  For a function $u$ to be continuous on $\Gamma$, in addition to the continuity of every  restriction $u_{e}$ on $J_{e}$, one has to require continuity  at the central vertex; that is, $u_{e}(0)=u_{r}(0)$  for $e$, $r=1$, $\ldots$, $N$. Given a function $F:\mathbb{C}\rightarrow \mathbb{R}$, if the integrand does not require more precision, we will abbreviate
\begin{equation*}
\sum^{N}_{e=1}\int_{\mathbb{R}^{+}}F(u_{e}(x))dx=\int_{\Gamma}F(u)dx.
\end{equation*}
Associated to a star graph $\Gamma$, we have a natural Hilbert space  $L^{2}(\Gamma)$, which is defined as the orthogonal direct sum of spaces $L^{2}(\mathbb{R}^{+})$. The space  $L^{2}(\Gamma)$ consists of functions that are in $L^{2}(\mathbb{R}^{+})$ for every edge  of $\Gamma$,  equipped with the norm given by
\begin{equation*}
\left\|u\right\|^{2}_{L^{2}(\Gamma)}=\int_{\Gamma}|u|^{2}dx=\sum^{N}_{e=1}\int_{\mathbb{R}^{+}}|u_{e}(x)|^{2}dx.
\end{equation*}
$L^{p}$-spaces on $\Gamma$ are defined analogously. The Sobolev space $H^{1}(\Gamma)$ on the graph $\Gamma$ consists of all continuous functions  $u=(u_{e})^{N}_{e=1}$ such that $u_{e}\in H^{1}(\mathbb{R}^{+})$. The continuity condition imposed on functions from the Sobolev space $H^{1}(\Gamma)$ means that any function $u$ from this space assumes the same value at the central vertex, and thus $u(0)$ is uniquely defined.  We say that  $u$ is symmetric if $u_{e}$ does not depends on $e$. For a general reference on analytical properties of functions defined over a graph, see the classical monograph \cite{QGAH}.

This paper is devoted to the analysis of  existence and stability of the ground states for the logarithmic Schr\"{o}dinger equation on a star graph $\Gamma$ with an attractive delta condition in the vertex,
\begin{equation}
\label{00NL}
 i\partial_{t}u+\Delta_{\gamma}u+u\, \mbox{Log}\left|u\right|^{2}=0, 
\end{equation}
where  $u$ is a complex-valued function of $(x,t)\in \Gamma\times\mathbb{R}$. Here, the nonlinear term in \eqref{00NL} is defined componentwise: namely, $\bigl(u\, \mbox{Log}|u|^{2}\bigr)_{i}=u_{i}\, \mbox{Log}|u_{i}|^{2}$  for $i=1$,$\ldots$,$N$.
 For $\gamma\in \mathbb{R}$, the Laplace operator $-\Delta_{\gamma}$ on the graph $\Gamma$ which appear in \eqref{00NL} admit a precise interpretation as self-adjoint operator on $L^{2}(\Gamma)$ associated with the quadratic form $\mathfrak{F}_{\gamma}$  (see \cite{KRQW}),
\begin{equation*}
\mathfrak{F_{\gamma}}[u]=\sum^{N}_{e=1}\int_{\mathbb{R}^{+}}|\partial_{x} u_{e}(x)|^{2}dx-\gamma \left|u_{1}(0)\right|^{2},
\end{equation*}
defined on the domain $\mathrm{dom}(\mathfrak{F_{\gamma}})=H^{1}(\Gamma)$. To be more specific, it is clear that this form is bounded from below and closed on $H^{1}(\Gamma)$.  Then the self-adjoint operator on $L^{2}(\Gamma)$ associated with $\mathfrak{F}_{\gamma}$ is  given by 
\begin{equation*}
 -\left(\Delta_{\gamma}u\right)_{i}=-\partial^{2}_{x}u_{i} \quad \text{for  $i=1$, $\ldots$, $N$},
\end{equation*}
on the domain 
\begin{equation}\label{FCC}
\mathrm{dom}(-\Delta_{\gamma})=\bigl\{u\in H^{1}(\Gamma): u_{i}\in H^{2}(\mathbb{R}^{+}), \quad \sum^{N}_{i=1}\partial^{}_{x}u_{i}(0)=-\gamma u_{1}(0)\bigr\}.
\end{equation}
When $\gamma=0$, the condition at the vertex \eqref{FCC} is usually referred to as the Kirchhoff's boundary condition. Notice that $-\Delta_{\gamma}$ generalizes to the graph the well know Schr\"{o}dinger operator with delta potential of strength $\gamma$ on the line. The following spectral properties of  $-\Delta_{\gamma}$ are known: $\sigma_{\rm ess}(-\Delta_{\gamma})=[0,\infty)$; if $\gamma\leq 0$, then $\sigma_{\rm p}(-\Delta_{\gamma})=\emptyset$; if $\gamma> 0$, then $\sigma_{\rm p}(-\Delta_{\gamma})=\left\{-\gamma^{2}/N^{2}\right\}$.

The nonlinear Schr\"{o}dinger equation \eqref{00NL} is formally associated with the energy functional $E$ defined by 
\begin{equation*}
E(u)=\frac{1}{2}\mathfrak{F_{\gamma}}[u]-\frac{1}{2}\int_{\Gamma}\left|u\right|^{2}\mbox{Log}\left|u\right|^{2}dx.
\end{equation*}
Unfortunately, due to the singularity of the logarithm at the origin, the functional fails to be finite as well of class $C^{1}$ on $ \mathrm{dom}(\mathfrak{F_{\gamma}})=H^{1}(\Gamma)$. Due  to  this  loss  of  smoothness,  it is convenient  to work in a suitable Banach space endowed with a Luxemburg type norm in order to make functional $E$  well defined and $C^{1}$ smooth.

Indeed, we will work with functions in  the Banach space (see Section \ref{S:0})
\begin{equation}\label{ASE}
{W}(\Gamma)=\bigl\{u\in H^{1}(\Gamma):\left|u_{e}\right|^{2}\mathrm{Log}\left|u_{e}\right|^{2}\in L^{1}(\mathbb{R}^{+}) \,  \text{for $e=1$, $\ldots$, $N$} \bigr\}.
\end{equation}
Then, we have that the energy functional $E$ is well-defined and of class $C^1$ on $W(\Gamma)$.

In \cite{AHAA}, when $\Gamma=\mathbb{R}$, by considering the line as a two-edge star graph, it is proved that the Cauchy problem for \eqref{00NL} is globally well-posed in $W(\mathbb{R})$. Moreover, it was shown in \cite{AHAA}  that there exists a unique positive (up to a phase) ground state and it is orbitally stable in the case where $\gamma>0$.  

The main aim of this paper is to extend the existence and stability results of \cite{AHAA} by considering  a $N$-edge star graph with $N\geq2$.  An analogous analysis is given  for the standard NLS equation on a star graph in \cite{AQFF}, and both are inspired by \cite{FO,ADNP}.

The next proposition gives a result on the existence of weak solutions to \eqref{00NL} in the energy space $W(\Gamma)$. The proof is contained in Section \ref{S:1}.

\begin{proposition} \label{PCS}
For any $u_{0}\in {W}(\Gamma)$, there is a unique maximal solution  $u\in C(\mathbb{R},{W}(\Gamma))\cap C^{1}(\mathbb{R}, {W}^{\prime}(\Gamma) )$  of \eqref{00NL}  such that $u(0)=u_{0}$ and $\sup_{t\in \mathbb{R}}\left\|u(t)\right\|_{{W}(\Gamma)}<\infty$. Furthermore, the conservation of energy and charge hold; that is, 
\begin{equation*}
E(u(t))=E(u_{0})\quad  and \quad \left\|u(t)\right\|^{2}_{L^{2}(\Gamma)}=\left\|u_{0}\right\|^{2}_{L^{2}(\Gamma)}\quad  \text{for all $t\in \mathbb{R}$}.
\end{equation*}
\end{proposition}
In the previous proposition, ${W}^{\prime}(\Gamma)$ is the dual space of ${W}^{}(\Gamma)$.  
A  standing wave solution of  \eqref{00NL}  is a solution of the form  $u(x,t)=e^{i\omega t}\varphi(x)$ where $\omega\in \mathbb{R}$ and $\varphi \in W(\Gamma)\setminus\{0\}$ is a real valued function which has to solve the following stationary  problem 
\begin{equation}\label{GS}
{-\Delta}_{\gamma}\varphi+\omega \varphi-\varphi\, \mathrm{Log}\left|\varphi \right|^{2}=0 \quad \mbox{in} \quad {W}^{\prime}(\Gamma).\\
\end{equation}

An explicit description of all the solutions of the stationary  problem \eqref{GS} is obtained for every value of $\gamma>0$. In fact, the stationary solutions to \eqref{GS} are given in the following result. We denote by  $\left[s\right]$  the integer part of $s$.

\begin{theorem} \label{ESTIO}
Let $N\geq2$, $\gamma>0$ and $\omega \in \mathbb{R}$. Then, the stationary problem \eqref{GS} has $\left[(N-1)/2\right]$ positive solutions $\phi^{\kappa}_{\omega,\gamma}$, with $\kappa=0$,  $\ldots$, $\left[(N-1)/2\right]$, given,  up to permutations of edges, by
\begin{align}\label{E5}
(\phi^{\kappa}_{\omega,\gamma})_{i}(x)&=
\begin{cases}
e^{\frac{\omega+1}{2}}e^{-\frac{1}{2}(x-h_{{\kappa}})^{2}}, &\text{$i=1$, $\ldots$, $\kappa$} \\
e^{\frac{\omega+1}{2}}e^{-\frac{1}{2}(x+h_{{\kappa}})^{2}}, &\text{$i=\kappa+1$, $\ldots$, $N$;} 
\end{cases}
\end{align}
where  $h_{{\kappa}}={\gamma}/{(N-2\kappa)}$.
\end{theorem}
The proof of this result is contained  in Section \ref{S:2}. We remark that the solution is unique only in the case $N=2$, that is, in the case of a two-edge star graph. 

The next step in the study of stationary solutions to \eqref{GS} is to understand their stability. To this aim, and when possible, we give a variational characterization of the stationary solutions.

For $\gamma>0$ and $\omega\in \mathbb{R}$, let us define the following functionals of class $C^{1}$ on $W(\Gamma)$:
\begin{align*}
 S_{\omega,\gamma}(u)&=\frac{1}{2}\mathfrak{F_{\gamma}}[u]+ \frac{\omega+1}{2}\|u \|^{2}_{L^{2}(\Gamma)}-\frac{1}{2}\int_{\Gamma}\left|u\right|^{2}\mathrm{Log}\left|u\right|^{2}dx,\\
 I_{\omega,\gamma}(u)&=\mathfrak{F_{\gamma}}[u]+\omega\,\|u \|^{2}_{L^{2}(\Gamma)}-\int_{\Gamma}\left|u\right|^{2}\mathrm{Log}\left|u\right|^{2}dx.
\end{align*}
Note that \eqref{GS} is equivalent to $S^{\prime}_{\omega,\gamma}(\varphi)=0$, and $I_{\omega,\gamma}(u)=\left\langle S_{\omega,\gamma}^{\prime}(u),u\right\rangle$ is the so-called Nehari functional. Moreover, we consider the minimization problem
\begin{align}
\begin{split}\label{MPE}
d_{\gamma}(\omega)&={\inf}\left\{S_{\omega,\gamma}(u):\, u\in W(\Gamma)  \setminus  \left\{0 \right\},  I_{\omega,\gamma}(u)=0\right\} \\ 
&=\frac{1}{2}\,{\inf}\left\{\left\|u\right\|_{L^{2}(\Gamma)}^{2}:u\in  W(\Gamma) \setminus \left\{0 \right\},  I_{\omega,\gamma}(u)= 0 \right\}, \end{split}
\end{align}
and define the set of ground states  by
\begin{equation*}
 \mathcal{G}_{\omega,\gamma}=\bigl\{ \varphi\in W(\Gamma)  \setminus  \left\{0 \right\}: S_{\omega,\gamma}(\varphi)=d_{\gamma}(\omega), \,\, I_{\omega,\gamma}(\varphi)=0\bigl\}.
\end{equation*}
\begin{remark}\label{RULT}
The set $\bigl\{u\in W(\Gamma) \setminus  \left\{0 \right\},  I_{\omega,\gamma}(u)=0\bigl\}$ is called the Nehari manifold. Since $I_{\omega,\gamma}(u)=\left\langle S_{\omega,\gamma}^{\prime}(u),u\right\rangle$, it clearly  contains all the nontrivial critical points of $ S_{\omega,\gamma}$. It is standard to show that if $u\in \mathcal{G}_{\omega,\gamma}$, then $u$ is a solution to the stationary equation \eqref{GS}.
\end{remark}
Before proceeding to our main results, we recall the definition of the error function 
\begin{equation}\label{erf}
\text{erf}(s)=\frac{2}{\sqrt{\pi}}\int^{s}_{0}e^{-t^{2}}dt \quad \text{for all $s\in \mathbb{R}^{}$}.
\end{equation}
We remark that the error function  is strictly monotonically increasing on $\mathbb{R}^{}$. We define the inverse error function as follows. For a positive $r$, if $r=\text{erf}(s)$ the inverse function $s=\text{erf}^{-1}(r)$. The domain $r$ for the inverse function is the interval $[0,1]$, and the range is $[0, +\infty)$.

 The existence of minimizers for \eqref{MPE} is obtained through variational argument. We will show the following theorem in Section \ref{S:3}.
\begin{theorem} \label{ESSW}
Let $N\geq2$, $\omega \in \mathbb{R}$ and  $\gamma^{\ast}(N):=N\left(\mathrm{erf}^{-1}(1-2/N)\right)$. Then, there exists a minimizer of $d_{\gamma}(\omega)$ for any $\gamma>\gamma^{\ast}(N)$. Moreover, the set of ground states is given by $\mathcal{G}_{\omega,\gamma}=\bigl\{e^{i\theta}\phi^{0}_{\omega,\gamma}: \theta\in\mathbb{R} \bigl\}$, where $\phi^{0}_{\omega,\gamma}$ is defined by \eqref{E5}. 
\end{theorem}
So $S_{\omega,\gamma}$ admits a constrained minimum on the Nehari manifold for every $\omega\in \mathbb{R}$ if the strength $\gamma$ of the $\delta$-interaction at the vertex is sufficiently strong. Moreover, in this case, any minimizing sequence of \eqref{MPE} is relatively compact in $W(\Gamma)$. On the other hand, if $0\leq\gamma<\gamma^{\ast}(N)$, then  the infimum $d_{\gamma}(\omega)$ is approximated by the action of a soliton (i.e. the ground state on the line) escaping to infinity. In particular, there exists a minimizing sequence of $d_{\gamma}(\omega)$ that converges weakly to the vanishing function; see the proof of Proposition \ref{LEM1} for more details. Notice that, for $N\geq2$, the function $N\rightarrow \gamma^{\ast}(N)$ is strictly monotonically increasing, $\gamma^{\ast}(2)=0$ and   $\gamma^{\ast}(N)\rightarrow +\infty$ as $N\rightarrow +\infty$. In particular, in the case of a two-edge star graph we have the following result which was proved in \cite{AHAA}: there exists a unique (up to a phase) ground state for all $\gamma>\gamma^{\ast}(2)=0$.  On the other hand, when $N\geq3$ and $0<\gamma\leq\gamma^{\ast}(N)$, it is conjectured that the action $S_{\omega,\gamma}$ has a local constrained minimum that is larger than the infimun, but we do not have a proof of this fact.  We note that the conjectured behavior has in fact been proved recently for the standard power nonlinearity by Adami-Cacciapuoti-Finco-Noja in \cite{CCP}.

It is important to note that the basic idea underlying this work as well as \cite{AQFF}, is that a ground state exists if and only if the action $S_{\omega,\gamma}$ of the unique symmetric stationary state $\phi^{0}_{\omega,\gamma}$ is lower than the action of the soliton  associated with the same frequency; see Sections \ref{S:15} and \ref{S:3} for a complete description. This explains the fact that the threshold $\gamma^{\ast}(2)$ for the two-edge graph equals zero.

Now we come to the stability of the ground state. The basic symmetry associated to equation \eqref{00NL} is the phase-invariance. Thus,  the definition  of stability  takes into account only  this  type of symmetry and is formulated as follows.

\begin{definition}\label{2D111}
We say that  a standing wave solution $u(x,t)=e^{i\omega t}\phi(x)$ of \eqref{00NL} is orbitally stable in $W(\Gamma)$ if for any  $\epsilon>0$ there exists $\eta>0$  such that if $u_{0}\in W(\Gamma)$ and $\left\|u_{0}-\varphi \right\|_{W(\Gamma)}<\eta$, then the solution $u(t)$ of  \eqref{00NL}  with $u(0)=u_{0}$ exist for all $t\in \mathbb R$ and satisfies 
\begin{equation*}
{\rm\sup\limits_{t\in \mathbb R}} {\rm\inf\limits_{\theta\in \mathbb{R}}} \|u(t)-e^{i\theta}\phi \|_{W(\Gamma)}<\epsilon.
\end{equation*}
Otherwise, the standing wave $e^{i\omega t}\phi(x)$ is said to be  unstable in $W(\Gamma)$.
\end{definition}
Making use of the arguments in \cite{AHAA,FO}, from the compactness of the minimizing sequences (see Lemma \ref{CSM} below)  and uniqueness of the ground states up to phase shown in Theorem \ref{ESSW}, the orbital stability of the ground states follows.
\begin{theorem} \label{EST}
Let $N\geq2$, $\omega \in \mathbb{R}$ and  $\gamma>\gamma^{\ast}(N)$.  Then the standing wave $e^{i\omega t}\phi^{0}_{\omega,\gamma}$ is orbitally stable in  $W(\Gamma)$.
\end{theorem}

We end this introduction with two remarks. Firstly, nothing rigorous  is known about orbital stability or instability of excited states, which exists for every $N\geq3$ and $\omega \in \mathbb{R}$; it is conjectured  that excited states are unstable, but we do not have a proof of this fact. On the other hand, an important breakthrough in the problem of determining the existence of ground states for Kirchhoff's graphs has come with the paper by Adami-Serra-Tilli \cite{ASTH}.  We claim that the techniques presented in that paper can be easily adapted  to the focusing logarithmic nonlinearity. Indeed, the rearrangements preserve the energy space $W(\Gamma)$ (see the proof of Proposition \ref{RSS}   below). Thus, implementing on graphs the rearrangement theory in a more thorough way, it is possible to give results for a general class of graphs (with no reason to limit to star-shaped only).

The rest of the paper is organized as follows. In Section \ref{S:0}, we analyse the structure of the energy space $W(\Gamma)$. Moreover, we recall several known results, which will be needed later. In Section \ref{S:1}, we give an idea of the proof of  Proposition \ref{PCS}. In Section \ref{S:2}, an explicit construction of all stationary states of problem is obtained. In Section \ref{S:15} we compute explicitly the infimum $d_{0}(\omega)$  (Kirchhoff's case), which will be a key ingredient for our analysis to follow. In Section \ref{S:3} we prove, by variational techniques, the existence of a minimizer of $d_{\gamma}(\omega)$ for any $\gamma>\gamma^{\ast}(N)$. We also explicitly compute the ground states (Theorem \ref{ESSW}). The Section \ref{S:4} is devoted to the proof of Theorem \ref{EST}. In the Appendix we show that the energy functional $E$ is of class $C^{1}$ on $W(\Gamma)$.\\

\noindent{\textbf{{Notation}}:} The space $L^{2}(\mathbb{R}^{+},\mathbb{C})$  will be denoted  by $L^{2}(\mathbb{R}^{+})$ and its norm by $\|\cdot\|_{L^{2}(\mathbb{R}^{+})}$.  This space will be endowed  with the real scalar product
\begin{equation*}
\left(u,v\right)=\Re\int_{\mathbb{R}^{+}}u\overline{v}\, dx \quad \mathrm{for }\quad u,v\in L^{2}\left(\mathbb{R}^{+}\right).
\end{equation*}
The space $H^{1}(\mathbb{R}^{+},\mathbb{C})$ will be denoted by $H^{1}(\mathbb{R}^{+})$, its norm by $\|\cdot\|_{H^{1}(\mathbb{R}^{+})}$. We denote by $C_{0}^{\infty}\left(\mathbb{R}^{+}\right)$ the set of $C^{\infty}$ functions from $\mathbb{R}^{+}$ to $\mathbb{C}$ with compact support. $\left\langle \cdot , \cdot \right\rangle$ is the duality pairing between $E^{\prime}$ and $E$, where $E$ is a Hilbert (more generally, Banach space) and $E^{\prime}$ is its dual.  Throughout this paper, the letter $C$ will denote positive constants.

\section{Preliminaries}
\label{S:0} In this section we analyse the structure of the energy space $W(\Gamma)$.  We also recall  several known results on the logarithmic Schr\"{o}dinger equation and the basic properties of the symmetric rearrangements on a star graph. 

\subsection{The energy space.} First we need to introduce some notation. Define 
\begin{equation*}
F(z)=\left|z\right|^{2}\mbox{Log}\left|z\right|^{2}\quad  \text{for every  $z\in\mathbb{C}$},
\end{equation*}
and as in \cite{CL},  we define the functions  $A$, $B$ on $\left[0, \infty\right)$  by 
\begin{equation}\label{IFD}
A(s)=
\begin{cases}
-s^{2}\,\mbox{Log}(s^{2}), &\text{if $0\leq s\leq e^{-3}$;}\\
3s^{2}+4e^{-3}s^{}-e^{-6}, &\text{if $ s\geq e^{-3}$;}
\end{cases}
\,\,\,\,\,\,\,\,\,  B(s)=F(s)+A(s).
\end{equation}
Furthermore, let be functions $a$, $b$, defined by
\begin{equation}\label{abapex}
a(z)=\frac{z}{|z|^{2}}\,A(\left|z\right|)\quad\text{ and  }\quad b(z)=\frac{z}{|z|^{2}}\,B(\left|z\right|)\quad \text{  for $z\in \mathbb{C}$, $z\neq 0$}.
\end{equation}
Notice that we have $b(z)-a(z)=z\,\mathrm{Log}\left|z\right|^{2}$.  It follows that $A$ is a nonnegative  convex and increasing function, and $A\in C^{1}\left([0,+\infty)\right)\cap C^{2}\left((0,+\infty)\right)$.  The Orlicz space $L^{A}(\mathbb{R}^{+})$ corresponding to $A$ is defined by
\begin{equation*}
L^{A}(\mathbb{R}^{+})=\left\{u\in L^{1}_{\rm loc}(\mathbb{R}^{+}) : A(\left|u\right|)\in L^{1}_{}(\mathbb{R}^{+})\right\}, 
\end{equation*}
equipped with the Luxemburg norm 
\begin{equation*}
\left\|u\right\|_{L^{A}(\mathbb{R}^{+})}={\inf}\left\{k>0: \int_{\mathbb{R}^{+}}A\left(k^{-1}{\left|u(x)\right|}\right)dx\leq 1 \right\}.
\end{equation*}
Here as usual $L^{1}_{\rm loc}(\mathbb{R}^{+})$ is the space of all locally Lebesgue integrable functions.

It is proved in \cite[Lemma 2.1]{CL} that $A$ is a Young-function which is $\Delta_{2}$-regular (see \cite[Chapter III]{ORL} for more details) and $\left(L^{A}(\mathbb{R}^{+}),\|\cdot\|_{L^{A}(\mathbb{R}^{+})} \right)$ is a separable reflexive  Banach space. Let us denote by $W(\mathbb{R}^{+})$ the reflexive Banach space $W(\mathbb{R}^{+})=H^{1}(\mathbb{R}^{+})\cap L^{A}(\mathbb{R}^{+})$ equipped with usual norm $\|\cdot\|_{W({\mathbb{R}^{+}})}$ defined by  $\|u\|_{W({\mathbb{R}^{+}})}=\|u\|_{H^{1}({\mathbb{R}^{+}})}+\|u\|_{L^{A}(\mathbb{R}^{+})}$.

Finally, we consider the reflexive Banach space 
\begin{equation*}
W(\Gamma)=\left\{u\in H^{1}(\Gamma):u_{e}\in W({\mathbb{R}^{+}}) \,  \text{for $e=1$, $\ldots$, $N$} \right\}, 
\end{equation*}
equipped with  norm 
\begin{equation*}
\left\|u\right\|^{2}_{{W}(\Gamma)}=\sum^{N}_{i=1}\left\|u_{i}\right\|^{2}_{W({\mathbb{R}^{+}})}.
\end{equation*}
It is easy to see that one has the following chain of continuous embedding $W(\Gamma)\hookrightarrow L^{2}(\Gamma)\hookrightarrow W^{\prime}(\Gamma)$.  Moreover, we have that
\begin{equation}\label{IUO}
{W}(\Gamma)=\bigl\{u\in H^{1}(\Gamma):\left|u_{e}\right|^{2}\mathrm{Log}\left|u_{e}\right|^{2}\in L^{1}({\mathbb{R}^{+}})\,\,   \text{for $e=1$, $\ldots$, $N$} \bigr\}.
\end{equation}
The proof of \eqref{IUO} follows immediately from analogous equality for functions of the real half line: namely, $W(\mathbb{R}^{+})=\bigl\{u\in H^{1}(\mathbb{R}^{+}):\left|u_{}\right|^{2}\mathrm{Log}\left|u_{}\right|^{2}\in L^{1}({\mathbb{R}^{+}}) \bigr\}$.  For the proof of the last statement we refer to Cazenave \cite[Proposition 2.2]{CL}.

The following remark will be useful later on.
\begin{remark}
For every $\epsilon>0$,  there exists $C_{\epsilon}>0$ such that 
\begin{equation*}
|B(z)-B(w)|\leq C_{\epsilon}\left(|z|^{1+\epsilon}+|w|^{1-\epsilon}\right)|z-w| \quad \text{for all $z$, $w\in \mathbb{C}$.}
\end{equation*}
Integrating the above inequality on  $\mathbb{R}^{+}$, applying H{\"o}lder's and Sobolev's inequalities  and summing on each edge of graph $\Gamma$, we deduce that for all $u$, $v\in H^{1}(\Gamma)$,
\begin{equation}\label{DB}
\int_{\Gamma}\left|B(\left|u\right|)- B(\left|v\right|)\right|dx\leq C\left(1+ \left\|u\right\|^{2}_{H^{1}(\Gamma)}+ \left\|v\right\|^{2}_{H^{1}(\Gamma)} \right)\left\|u-v\right\|_{{L^{2}(\Gamma)}}.
\end{equation}
\end{remark}
We list some properties of the Orlicz space $L^{A}(\mathbb{R}^{+})$, which will be needed later.  For a proof of such statements we refer to \cite[Lemma 2.1]{CL}.

\begin{proposition} \label{orlicz}
Let $\left\{u_{{m}}\right\}$ be a sequence in  $L^{A}(\mathbb{R}^{+})$, the following facts hold:\\
{\it i)} If  $u_{{m}}\rightarrow u$ in $L^{A}(\mathbb{R}^{+})$, then $A(\left|u_{{m}}\right|)\rightarrow A(\left|u\right|)$ in $L^{1}(\mathbb{R}^{+})$ as   $n\rightarrow \infty$.\\
{\it ii)} Let  $u\in L^{A}(\mathbb{R}^{+})$. If  $u_{m}(x)\rightarrow u(x)$ $a.e.$  $x\in\mathbb{R}^{+}$ and if 
\begin{equation*}
\lim_{n \to \infty}\int_{\mathbb{R}^{+}}A\left(\left|u_{m}(x)\right|\right)dx=\int_{\mathbb{R}^{+}}A\left(\left|u(x)\right|\right)dx,
\end{equation*}
then $u_{{m}}\rightarrow u$ in $L^{A}(\mathbb{R}^{+})$ as   $n\rightarrow \infty$.\\
{\it iii)} For any $u\in L^{A}(\mathbb{R}^{+})$, we have
\begin{equation}\label{DA1}
{\rm min} \left\{\left\|u\right\|_{L^{A}(\mathbb{R}^{+})},\left\|u\right\|^{2}_{L^{A}(\mathbb{R}^{+})}\right\}\leq  \int_{\mathbb{R}^{+}}A\left(\left|u(x)\right|\right)dx\leq {\rm max} \left\{\left\|u\right\|_{L^{A}(\mathbb{R}^{+})},\left\|u\right\|^{2}_{L^{A}(\mathbb{R}^{+})}\right\}.
\end{equation}
\end{proposition}

\subsection{ Variational characterization of the ground state on the half-line and  on the line.}  We recall a well-known result on the logarithmic Schr\"{o}dinger equation on the line: namely, the  set of  solutions of the stationary problem (see \cite[Appendix D]{CAS})
\begin{equation*}
 -\partial^{2}_{x} \varphi+\omega \varphi-\varphi\,\mathrm{Log}\left|\varphi \right|^{2}=0, \quad \text{ $x\in\mathbb{R}$,\, $\omega\in\mathbb{R}$,\, $\varphi\in W(\mathbb{R})$},
\end{equation*}
is given by $\bigl\{e^{i\theta}\phi_{\omega}(\cdot-y); \theta\in\mathbb{R},  y\in\mathbb{R} \bigl\}$, where
\begin{equation}\label{AUXF}
\phi_{\omega}(x)=e^{\frac{\omega+1}{2}}e^{-\frac{1}{2}x^{2}}.
\end{equation}
In addition, the soliton $\phi_{\omega}$ is the only minimizer (modulo translation and phase) of  problem
\begin{align}
\begin{split}\label{MPEaa}
d_{\mathbb{R}}(\omega)&={\inf}\left\{S_{\mathbb{R}}(u,\omega):\, u\in W(\mathbb{R})\setminus  \left\{0 \right\},  I_{\mathbb{R}}(u,\omega)=0\right\},\\ 
\end{split}
\end{align}
where
\begin{align*}
 S_{\mathbb{R}}(u,\omega)&=\frac{1}{2} \|\partial^{}_{x}u \|^{2}_{L^{2}(\mathbb{R})}+ \frac{\omega+1}{2}\|u \|^{2}_{L^{2}(\mathbb{R})}-\frac{1}{2}\int_{\mathbb{R}}\left|u\right|^{2}\mbox{Log}\left|u\right|^{2}dx,\\
 I_{\mathbb{R}}(u,\omega)&= \|\partial^{}_{x}u \|^{2}_{L^{2}(\mathbb{R})}+ \omega\,\|u \|^{2}_{L^{2}(\mathbb{R})}-\int_{\mathbb{R}}\left|u\right|^{2}\mbox{Log}\left|u\right|^{2}dx.\end{align*}
Moreover $d_{\mathbb{R}}(\omega)=e^{\omega+1}\sqrt{\pi}/2$.  For the proof of this result we refer to A.H. Ardila \cite{AHA1}. This  implies that  half soliton  $\chi_{+}\phi_{\omega}$ is the solution of the problem 
\begin{align}
\begin{split}\label{MPEaame}
d_{\mathbb{R}^{+}}(\omega)&={\inf}\left\{S_{\mathbb{R}^{+}}(u,\omega):\, u\in W(\mathbb{R}^{+})\setminus  \left\{0 \right\},  I_{\mathbb{R}^{+}}(u,\omega)=0\right\},\\ 
\end{split}
\end{align}
where
\begin{align*}
 S_{\mathbb{R}^{+}}(u,\omega)&=\frac{1}{2} \|\partial^{}_{x}u \|^{2}_{L^{2}(\mathbb{R}^{+})}+ \frac{\omega+1}{2}\|u \|^{2}_{L^{2}(\mathbb{R}^{+})}-\frac{1}{2}\int_{\mathbb{R}^{+}}\left|u\right|^{2}\mbox{Log}\left|u\right|^{2}dx,\\
 I_{\mathbb{R}^{+}}(u,\omega)&= \|\partial^{}_{x}u \|^{2}_{L^{2}(\mathbb{R}^{+})}+ \omega\,\|u \|^{2}_{L^{2}(\mathbb{R}^{+})}-\int_{\mathbb{R}^{+}}\left|u\right|^{2}\mbox{Log}\left|u\right|^{2}dx.\end{align*}
Moreover, we have that $d_{\mathbb{R}^{+}}(\omega)=d_{\mathbb{R}}(\omega)/2$. To prove the last statement, assume that $u\in W(\mathbb{R}^{+})\setminus\left\{0\right\}$ is such that $I_{\mathbb{R}^{+}}(u,\omega)=0$ and 
\begin{equation*}
S_{\mathbb{R}^{+}}(u,\omega)\leq S_{\mathbb{R}^{+}}(\chi_{+}\phi_{\omega},\omega).
\end{equation*}
Then, denoted by $\hat{u}$ the even extension of $u$, we see that $I_{\mathbb{R}}(\hat{u},\omega)=0$ and 
\begin{equation*}
S_{\mathbb{R}^{}}(\hat{u},\omega)\leq S_{\mathbb{R}^{}}(\phi_{\omega},\omega).
\end{equation*}
Thus, since $\hat{u}$ is even  and $\phi_{\omega}$ is the only minimizer (modulo translation and phase) of  problem \eqref{MPEaa}, we infer that $\hat{u}$ must be equal to  $\phi_{\omega}$ up a phase factor.

\subsection{Symmetric rearrangements.} In this subsection we recall the basic properties  of symmetric rearrangements $u^{\ast}$  of a measurable function $u:\Gamma\rightarrow \mathbb{C}^{N}$, where $\Gamma$ is a star graph.  

Given $u:\Gamma\rightarrow \mathbb{C}^{N}$, we introduce $\lambda_{u}(s)$ and $\varsigma_{u}(s)$ defined by
\begin{equation*}
\lambda_{u}(s)=\left|\left\{\left|u\right|\geq s\right\}\right| \quad \text{and} \quad \varsigma_{u}(s)=\sup\left\{s| \lambda_{u}(s)>Nt\right\},
\end{equation*}
and as in \cite{AQFF}, we define the symmetric rearrangement $u^{\ast}$ of $u$ by $u^{\ast}=(u^{\ast}_{1}, \ldots, u^{\ast}_{N})$ with 
\begin{equation*}
u^{\ast}_{1}(x)=\ldots=u^{\ast}_{N}(x)=\varsigma_{u}(x).
\end{equation*}

The basic properties of the function $u^{\ast}$  are given in the following proposition.
\begin{proposition} \label{RSS}
Let  $u\in H^{1}(\Gamma)$. Then   the following assertions hold.\\
{\rm (i)} The  symmetric rearrangement $u^{\ast}$ is positive, symmetric and non increasing. Moreover,  $u^{\ast}\in H^{1}(\Gamma)$, 
$\left\|u^{\ast}\right\|_{L^{p}(\Gamma)}=\left\|u\right\|_{L^{p}(\Gamma)}$ and $\left\|\partial^{}_{x}u^{\ast}\right\|_{L^{2}(\Gamma)}\leq (N/2)\left\|\partial^{}_{x}u^{}\right\|_{L^{2}(\Gamma)}$.\\
{\rm (ii)} If $u\in W(\Gamma)$, then $u^{\ast}\in W(\Gamma)$ and 
\begin{equation*}
\int_{\Gamma}\left|u^{\ast}\right|^{2}\mathrm{Log}\left|u^{\ast}\right|^{2}dx=\int_{\Gamma}\left|u\right|^{2}\mathrm{Log}\left|u\right|^{2}dx.
\end{equation*}
\end{proposition}
\begin{proof}
The proof of ${\rm (i)}$ is contained in Proposition A.1 and Theorem 6 of \cite{AQFF}.
Now we prove ${\rm (ii)}$. We first recall that, by \eqref{IFD}, $\left|z\right|^{2}\mbox{Log}\left|z\right|^{2}=A(\left|z\right|)-B(\left|z\right|)$  for every  $z\in\mathbb{C}$. Moreover, since  $u\in W(\Gamma)$, it follows that $A(\left|u_{e}\right|)\in L^{1}(\mathbb{R}^{+})$ and $B(\left|u_{e}\right|)\in L^{1}(\mathbb{R}^{+})$ for $e=1$, $\ldots$, $N$. Here, $u_{e}$ is the restriction of $u$ on the edge $J_{e}$.   As it was observed in \cite[Proposition A.1]{AQFF}, the symmetric rearrangement is equimeasurable, that is, 
\begin{equation}\label{EQS}
\left|\left\{\left|u\right|\geq s\right\}\right|=\left|\left\{u^{\ast}\geq s\right\}\right|.
\end{equation}
Since $A \in C^{1}(\overline{\mathbb{R}}^{+})$ is an increasing function with $A(0)=0$, it easily follows from Layer cake representation \cite[Theorem 1.13]{ELL} and \eqref{EQS}  that
\begin{align}\nonumber
\int_{\Gamma}A(\left|u\right|)dx&= \sum^{N}_{i=1}\int_{\mathbb{R}^{+}}A(\left|u_{i}(x)\right|)dx=\int^{+\infty}_{0} A^{\prime}(s)\left|\left\{\left|u\right|\geq s\right\}\right|ds\\\nonumber
&=\int^{+\infty}_{0} A^{\prime}(s)\left|\left\{u^{\ast}\geq s\right\}\right|ds=\sum^{N}_{i=1}\int_{\mathbb{R}^{+}}A(\left|u^{\ast}_{i}(x)\right|)dx\\\label{AC}
&=\int_{\Gamma}A(\left|u^{\ast}\right|)dx.
\end{align}
Similarly, since $B \in C^{1}(\overline{\mathbb{R}}^{+})$ is an increasing function with $B(0)=0$, by applying the same argument as above we see that
\begin{equation}\label{BC}
\int_{\Gamma}B(\left|u\right|)dx=\sum^{N}_{i=1}\int_{\mathbb{R}^{+}}B(\left|u_{i}(x)\right|)dx=\sum^{N}_{i=1}\int_{\mathbb{R}^{+}}B(\left|u^{\ast}_{i}(x)\right|)dx=\int_{\Gamma}B(\left|u^{\ast}\right|)dx.
\end{equation}
In particular, $A(|u^{\ast}_{e}|)\in L^{1}(\mathbb{R}^{+})$ and  $B(|u^{\ast}_{e}|)\in L^{1}(\mathbb{R}^{+})$ for $e=1$, $\ldots$, $N$. Therefore, $u^{\ast}\in W(\Gamma)$ and the result follows from \eqref{AC} and  \eqref{BC}.
\end{proof}

 \section{The Cauchy problem}
\label{S:1}
In this section we sketch the proof of the global well-posedness of the Cauchy Problem  for \eqref{00NL} in the energy space  ${W}(\Gamma)$. The proof of Proposition \ref{PCS} is an adaptation of the proof of \cite[Theorem 9.3.4]{CB} (see also \cite{AHAA}).  So, we will approximate the logarithmic nonlinearity by a smooth nonlinearity, and as a consequence we construct a sequence of global solutions of the regularized Cauchy problem in $C(\mathbb{R},H^{1}(\Gamma))\cap C^{1}(\mathbb{R},H^{-1}(\Gamma))$,  then we pass to the limit using standard compactness results, extract a subsequence which converges to the solution of the limiting equation \eqref{00NL}. Finally, by using special properties  of the logarithmic nonlinearity we establish  uniqueness of the global solution.

Before outlining the main ideas of the proof of Proposition \ref{PCS}, we first need to introduce some notation. Let $\Gamma_{k}$ be a  compact star graph consisting of a central vertex $c$ and $N$ edges  attached to it,  where each edge $e$ of $\Gamma_{k}$ is associated with a open bounded interval $J_{e}=(0,k)$ of length $k>0$ and zero corresponds to the central vertex $c$. Let us  recall that the Sobolev space $H^{1}_{}(\Gamma_{k})$ consists of all continuous functions $u=(u_{e})^{N}_{e=1}$ such that $u_{e}\in H^{1}(0,k)$ for $e=1$, $\ldots$, $N$. We denote with $C_{0}(\Gamma_{k})$ the space of all complex-valued, continuous functions on the graph $\Gamma_{k}$, which tend to zero near all of the outer vertices. Furthermore, the Sobolev space $H^{1}_{0}(\Gamma_{k})$ on the graph $\Gamma_{k}$ consists of all functions $f\in C_{0}(\Gamma_{k})$ such that $f_{e}\in H^{1}(0,k)$ for every $e=1$, $\ldots$, $N$. As usual, 
it follows  that the inclusion map $H^{1}_{0}(\Gamma_{k})\hookrightarrow H^{1}_{}(\Gamma)$ is continuous. The dual space
of $H^{1}_{0}(\Gamma_{k})$ will be denoted by $H^{-1}(\Gamma_{k})$. 

First we regularize the logarithmic nonlinearity near the origin.  For $z\in \mathbb{C}$ and $m\in \mathbb{N}$,  we define the functions $a_{m}$ and $b_{m}$ by 
\begin{equation*}
a_{m}(z)=
\begin{cases}
 a_{}(z), &\text{if $\left|z\right|\geq \frac{1}{m}$;}\\
m\,z\,a_{}(\frac{1}{m}) , &\text{if $\left|z\right|\leq \frac{1}{m}$;}
\end{cases}
\quad \text{and} \quad 
b_{m}(z)=
\begin{cases}
 b_{}(z) , &\text{if $\left|z\right|\leq {m}$;}\\
\frac{z}{m}\,b({m}) , &\text{if $\left|z\right|\geq {m}$,}
\end{cases}
\end{equation*}
where   $a$ and $b$ were defined in \eqref{abapex}. For any fixed $m\in \mathbb{N}$,  we define a family of regularized nonlinearities in the form  $g_{m}(z)=b_{m}(z)-a_{m}(z)$,  for every $z\in \mathbb{C}$.

In order to construct a solution of  \eqref{00NL}, we solve first, for $m\in \mathbb{N}$, the regularized Cauchy problem
\begin{equation}\label{AHAX}
i\partial_{t}u^{m}+\Delta_{\gamma}u^{m}+g_{m}(u^{m})=0.\\
\end{equation}
\begin{proposition} \label{APCS}
For any $u_{0}\in H^{1}(\Gamma)$, there is a unique  solution  $u\in C(\mathbb{R},H^{1}(\Gamma))\cap C^{1}(\mathbb{R}, H^{-1}(\Gamma))$ of \eqref{AHAX}  such that $u(0)=u_{0}$. Furthermore, the conservation of energy and charge hold; that is, 
\begin{equation*}
\mathcal{E}_{m}(u^{m}(t))=\mathcal{E}_{m}(u_{0})\quad  and \quad \left\|u^{m}(t)\right\|^{2}_{L^{2}(\Gamma)}=\left\|u_{0}\right\|^{2}_{L^{2}(\Gamma)}\quad  \text{for all $t\in \mathbb{R}$},
\end{equation*}
where 
\begin{equation*}
\mathcal{E}_{m}(u)=\frac{1}{2}\mathfrak{F_{\gamma}}[u]-\frac{1}{2}\int_{\Gamma}G_{m}(u)dx, \quad G_{m}(z)=\int^{\left|z\right|}_{0}g_{m}(s)ds.
\end{equation*}
\end{proposition}
\begin{proof}
Since $g_{m}$ is globally Lipschitz continuous $\mathbb{C}\rightarrow \mathbb{C}$, the global well-posedness in
$H^{1}(\Gamma)$ and the conservation laws are well known, and follow from the standard fixed point argument and Gronwall lemma; see \cite{AQFF} for an exhaustive treatment in the case of NLS equation with a power nonlinearity $|u|^{p-1}u$. 
\end{proof}

For the proof of Proposition \ref{PCS}, we will use the following lemma.
\begin{lemma}\label{3ACS} 
Let $\left\{u^{{m}}\right\}_{m\in\mathbb{N}}$ be a bounded sequence in $L^{\infty}(\mathbb{R},H^{1}_{}(\Gamma))$. If $(u|_{\Gamma_{k}})_{m\in\mathbb{N}}$ is a  bounded sequence of  $W^{1,\infty}(\mathbb{R}, H^{-1}(\Gamma_{k}))$ for $k\in\mathbb{N}$, then there exists a subsequence, which we still denote by $\left\{u^{{m}}\right\}_{m\in\mathbb{N}}$, and there exists  $u\in L^{\infty}(\mathbb{R},H^{1}(\Gamma))$ for every $k\in\mathbb{N}$, such that the following properties hold:\\
{\rm (i)}$\left.u\right|_{\Gamma_{k}}\in W^{1,\infty}(\mathbb{R},H^{-1}(\Gamma_{k}))$ for every $k\in\mathbb{N}$.\\
{\rm (ii)}  $u^{m}(t)\rightharpoonup u^{}(t)$  in  $H^{1}(\Gamma)$  as $m\rightarrow \infty$ for every $t\in \mathbb{R}$.\\
{\rm (iii)}	For every $t\in\mathbb{R}$ there exists a subsequence  $m_{j}$ such that $u_{e}^{m_{j}}(x,t)\rightarrow u_{e}^{}(x,t)$   as $j\rightarrow \infty$, for a.e. $x\in \mathbb{R}^{+}$ and  $e=1$, $\ldots$, $N$.\\
{\rm (iv)}	$u_{e}^{m_{}}(x,t)\rightarrow u_{e}^{}(x,t)$ as $m\rightarrow \infty$,  for a.e.  $(x,t)\in \mathbb{R}^{+}\times\mathbb{R}$ and $e=1$, $\ldots$, $N$. 
\end{lemma}
\begin{proof}
We just sketch the proof since it follows the same ideas as the proof  of Lemma 9.3.6 in \cite{CB}. In fact, fix $k\in\mathbb{N}$. Note that 
$\left\{\left.u^{m}\right|_{\Gamma_{k}}\right\}_{m\in\mathbb{N}}$ is a bounded sequence of $ L^{\infty}((-k,k),H^{1}(\Gamma_{k}))\cap W^{1,\infty}((-k,k),H^{-1}(\Gamma_{k}))$. Therefore, by \cite[Proposition 1.1.2]{CB} there exists a subsequence, which we still denote by $\left\{u^{{m}}\right\}_{m\in\mathbb{N}}$, and there exists  $u\in L^{\infty}((-k,k),H^{1}(\Gamma_{k}))$ such that $\left.u^{m}(t)\right|_{\Gamma_{k}}\rightharpoonup u^{}(t)$ in $H^{1}(\Gamma_{k})$ for all $t\in (-k,k)$. Letting $k\rightarrow+\infty$ and  considering diagonal sequence, we see that there exists  $u\in L^{\infty}(\mathbb{R},H^{1}(\Gamma_{}))$ such that  $u^{m}(t)\rightharpoonup u^{}(t)$  in  $H^{1}(\Gamma)$  for every $t\in \mathbb{R}$. Thus,  $u\in L^{\infty}(\mathbb{R},H^{1}(\Gamma))$, and (ii) follows.  In addition, by \cite[ Remark 1.3.13(ii)]{CB} and (ii), we have that  $u\in W^{1,\infty}(\mathbb{R},H^{-1}(\Gamma_{k}))$ for every $k\in\mathbb{N}$. Hence (i) is established. The remainder of the proof follows similarly to the remainder of the proof  of \cite[Lemma 9.3.6]{CB}. 
\end{proof}

\begin{proof}[ \bf {Proof of Proposition \ref{PCS}}] Our proof follows the ideas of Cazenave \cite[Theorem 9.3.4]{CB}. 
Applying Proposition \ref{APCS}, we see that  for every $m\in \mathbb{N}$ there exists a unique global solution $u^{m}\in C(\mathbb{R}, H^{1}(\Gamma))\cap C^{1}(\mathbb{R},  H^{-1}(\Gamma))$ of \eqref{AHAX}, which satisfies 
\begin{equation}\label{JKL}
\mathcal{E}_{m}(u^{m}(t))=\mathcal{E}_{m}(u_{0})\quad \mbox{and}\quad\left\|u^{m}(t)\right\|^{2}_{L^{2}}=\left\|u_{0}\right\|^{2}_{L^{2}}  \,\, \text{ for all $t\in \mathbb{R}$},
\end{equation}
where \begin{equation*}
\mathcal{E}_{m}(u)=\frac{1}{2}\mathfrak{F_{\gamma}}[u] +\frac{1}{2}\int_{\Gamma}\Phi_{m}(\left|u_{}\right|)dx-\frac{1}{2}\int_{\Gamma}\Psi_{m}(\left|u_{}\right|)dx,
\end{equation*}
and the functions $\Phi_{m}$ and $\Psi_{m}$  defined by
\begin{equation*}
\Phi_{m}(z)=\frac{1}{2}\int^{\left|z\right|}_{0}a_{m}(s)ds \quad \mbox{and} \quad \Psi_{m}(z)=\frac{1}{2}\int^{\left|z\right|}_{0}b_{m}(s)ds.
\end{equation*}
It follows from \eqref{JKL} that $u^{m}$ is bounded in $L^{\infty}(\mathbb{R}, L^{2}(\Gamma))$. Moreover, we have that the sequence of approximating solutions $u^{m}$ is bounded in the space $L^{\infty}(\mathbb{R}, H^{1}(\Gamma))$  (see proof of Step 2 of \cite[Theorem 9.3.4]{CB}). It also follows from the NLS equation \eqref{AHAX} that the sequence $\left.\partial_{t}u^{m}\right|_{\Gamma_{k}}$ is bounded in the space $L^{\infty}(\mathbb{R}, H^{-1}(\Gamma_{k}))$. Therefore, we have that  $\left\{u^{{m}}\right\}_{m\in\mathbb{N}}$ satisfies the assumptions of  Lemma \ref{3ACS}. Let $u$ be the limit of $u^{m}$.

Now we show that the limiting function $u\in L^{\infty}(\mathbb{R},H^{1}(\Gamma))$ is a weak solution of the logarithmic NLS equation \eqref{00NL}. To do so, we first write  a weak formulation of the NLS equation \eqref{AHAX}. Indeed, for any test continuous function $\psi:=(\psi_{e})^{N}_{e=1}$ with $\psi_{e}\in C^{\infty}_{0}([0,+\infty))$  and $\phi\in C^{\infty}_{0}({\mathbb{R}})$, we have
\begin{equation}\label{3DPL}
-\int_{\mathbb{R}}\left[\left\langle i\, u^{m}, \psi\right\rangle \phi^{\prime}(t)+\mathfrak{F_{\gamma}}[u^{m},\psi] \phi^{}(t)\right]\,dt+\int_{\mathbb{R}}\Big(\sum^{N}_{i=1}\int_{\mathbb{R}^{+}}g_{m}(u_{i}^{m})\psi_{i}(x)\,dx\Big)\phi(t)\,dt=0.
\end{equation}
Furthermore, since $g_{m}(z)\rightarrow z\, \mbox{Log}\left|z\right|^{2}$ pointwise in $z\in \mathbb{C}$ as $m\rightarrow+\infty$, we apply the properties (ii)-(iv) of Lemma \eqref{3ACS} to the integral formulation \eqref{3DPL} and obtain the following integral equation
 \begin{equation}\label{ert}
-\int_{\mathbb{R}}\left[\left\langle i\, u, \psi\right\rangle \phi^{\prime}(t)+\mathfrak{F_{\gamma}}[u,\psi] \phi^{}(t)\right]\,dt+\int_{\mathbb{R}}\Big(\sum^{N}_{i=1}\int_{\mathbb{R}^{+}}u_{i}\, \mbox{Log}\left|u_{i}\right|^{2}\psi_{i}(x)\,dx\Big)\phi(t)\,dt=0.
\end{equation}
In addition, $u(0)=u_{0}$ by property (ii) of Lemma \ref{3ACS}. Moreover,  it is easy to see that $u\in{L^{\infty}(\mathbb{R}, W(\Gamma))}$ (see proof of Step 3 of \cite[Theorem 9.3.4]{CB}). Therefore, by integral equation \eqref{ert},  $u\in{L^{\infty}(\mathbb{R}, W(\Gamma))}$ is a weak solution of the logarithmic NLS equation \eqref{00NL}. In particular, from Lemma \ref{APEX23} in Appendix, we deduce that $u\in W^{1,\infty}(\mathbb{R}, W^{\prime}(\Gamma))$.

Now we show uniqueness of the solution in the class $L^{\infty}(\mathbb{R}, W^{}(\Gamma))\cap W^{1,\infty}(\mathbb{R}, W^{\prime}(\Gamma))$. Indeed, let $u$ and $v$ be two solutions   of \eqref{00NL} in that class. Then $u-v$ satisfies a weak formulation similar  to the integral equation \eqref{ert} for the partial differential equation
\begin{equation*}
i\partial_{t}(u-v)+\Delta_{\gamma}(u-v)+(u\, \mbox{Log}\left|u\right|^{2}-v\, \mbox{Log}\left|v\right|^{2})=0. 
\end{equation*}
Multiplying this equation by $i(u-v)$ and integrating over $\Gamma$, we have
\begin{equation*}
\frac{d}{dt}\left\|u(t)-v(t)\right\|^{2}_{L^{2}(\Gamma)}=-\Im \sum^{N}_{i=1}\int_{\mathbb{R}^{+}}\left(u_{i}\mathrm{Log}\left|u_{i} \right|^{2}-v_{i}\mathrm{Log}\left|v_{i} \right|^{2} \right)(\overline{u_{i}}-\overline{v_{i}})dx.
\end{equation*}
Thus, from \cite[Lemma 9.3.5]{CB} we obtain
\begin{equation*}
\left\|u(t)-v(t)\right\|_{L^{2}(\Gamma)}\leq 8\int^{t}_{0}\left\|u(s)-v(s)\right\|^{2}_{L^{2}(\Gamma)}ds.
\end{equation*}
Therefore, the uniqueness of the solution follows by Gronwall's Lemma.

We claim that the weak solution $u$ of the  logarithmic NLS equation \eqref{00NL} satisfies both conservation of charge and energy. Indeed, by weak lower semicontinuity of the $H^{1}(\Gamma)$-norm, Fatou's lemma and arguing in the same way as in the proof of the Step 3 of \cite[Theorem 9.3.4]{CB} we deduce that
\begin{equation}\label{LON}
E(u(t))\leq E(u_{0}) \quad \mbox{and} \quad \left\|u^{}(t)\right\|^{2}_{L^{2}(\Gamma)}=\left\|u_{0}\right\|^{2}_{L^{2}(\Gamma)} \quad \text{for all} \,\, t\in \mathbb{R}.
\end{equation}
Now fix $t_{0}\in \mathbb{R}$. Let $\varphi=u(t_{0})$ and let $w$ be the solution of \eqref{00NL} with $w(0)=\varphi$. By uniqueness, we see that $w(\cdot-t_{0})=u(\cdot)$ on $\mathbb{R}$. From \eqref{LON}, we deduce in particular that  
\begin{equation*}
E(u_{0})\leq E(\varphi).
\end{equation*}
Therefore, we have that both $\left\|u^{}(t)\right\|^{2}_{L^{2}(\Gamma)}$ and $E(u(t))$ are constant on $\mathbb{R}$. Finally, the continuity of the solution $u\in C(\mathbb{R}, W(\Gamma))\cap C^{1}(\mathbb{R}, W^{\prime}(\Gamma))$ in time $t$ follow from the arguments identical to the case of the logarithmic NLS equation on $\mathbb{R}^{N}$ (see proof of Step 4 of \cite[Theorem 9.3.4]{CB}). 
\end{proof}

\section{Stationary Problem}
\label{S:2}
The aim of this section is to prove Theorem \ref{ESTIO}. Some preparation is needed.

By elliptic regularity, the solutions of the stationary problem \eqref{GS} are in fact smooth on each edge, except at the vertex, where they satisfy the boundary condition.
\begin{lemma} \label{LCU}
Let $\gamma\in\mathbb{R}\setminus \left\{0\right\}$,  $\omega\in \mathbb{R}$ and $u\in W(\Gamma)$ be a solution of \eqref{GS}. Then, for $l=1$, $\ldots$, $N$, the restriction $u_{l}:\mathbb{R}^{+}\rightarrow \mathbb{C}$ of $u$ to the $l$-th edge verifies the following:
\begin{align}
& u_{l}\in C^{2}(\mathbb{R}^{+}),  \label{1} \\
& - \partial^{2}_{x}u_{l}+\omega u_{l}-u_{l}\,  \mathrm{Log}\left|u_{l} \right|^{2}=0 \quad \mbox{on}\quad \mathbb{R}^{+},\label{2}  \\
& \partial_{x}u_{l}(x), u_{l}(x)\rightarrow 0, \quad  \mbox{as} \quad x\rightarrow\infty. \label{3}
\end{align}
Moreover, $u$ satisfies the jump condition $\partial_{x}u_{1}(0)+\ldots+\partial_{x}u_{N}(0)=-\gamma u_{1}(0)$.
\end{lemma}
\begin{proof}
Fix $l\in \left\{1,\ldots, N\right\}$. The proof of item \eqref{1} follow by a standard bootstrap argument using test functions $\xi \in C^{\infty}_{0}(\mathbb{R}^{+})$ (see e.g. \cite[Chapter 8]{CB}). Indeed, from \eqref{GS} applied with $\psi=(\psi_{i})^{N}_{i=1}$, where $\psi_{l}=\xi$ and $\psi_{i}=0$ for $i\neq l$,  we deduce that 
\begin{equation}\label{3ALU}
-\partial^{2}_{x}(\xi u_{l}) +\omega \xi\, u_{l}=-\partial^{2}_{x}\xi\, u_{l}-2 \partial^{}_{x}\xi\, \partial^{}_{x}u_{l}+\xi \,u_{l}\, \mbox{Log}\left|u_{l}\right|^{2},
\end{equation}
in the sense of distributions on $\mathbb{R}^{+}$. The right hand side is in ${L^{2}(\mathbb{R}^{+})}$  and so $\xi u_{l}\in H^{2}(\mathbb{R}^{+})$. This implies that $u_{l}$ is in $C^{2}(\mathbb{R}^{+})\cap H_{\rm loc}^{2}(\mathbb{R}^{+})$ and is a classical solution of this equation on $\mathbb{R}^{+}$,  from which \eqref{1} and \eqref{2} follows. Moreover, since $u_{l}\in H^{1}(\mathbb{R}^{+})$, it follows that $u_{l}(x)\rightarrow 0$ as $x\rightarrow \infty$. Thus, by \eqref{2}, $\partial^{2}_{x}u_{l}(x)\rightarrow 0$ as $x\rightarrow \infty$, and so $\partial^{}_{x}u_{l}(x)\rightarrow 0$ as $x\rightarrow \infty$. Finally, we consider any test continuous function on $\Gamma$,  $\varphi=(\varphi_{i})^{N}_{i=1}$ with $\varphi_{i}(0)=1$. From \eqref{GS}, we see that 
\begin{equation}\label{SGRT}
\mathfrak{F_{\gamma}}[u, \varphi]+\sum^{N}_{i=1}\left\{\omega\int_{\mathbb{R}^{+}}u_{i}\varphi_{i}dx-\int_{\mathbb{R}^{+}}u_{i}\,  \mathrm{Log}\left|u_{i} \right|^{2}\varphi_{i}dx\right\}=0.
\end{equation}
Thus, starting from \eqref{SGRT} and using \eqref{2} after an integration by parts gives the jump condition at the vertex, which concludes the proof.
\end{proof}

\begin{lemma} \label{L8}
Let  $u\in W(\mathbb{R}^{+})\cap C^{2}(\mathbb{R}^{+})$ be a non-trivial classical solution of
\begin{equation}\label{A22}
-\partial^{2}_{x}u+\omega u-u\,  \mathrm{Log}\left|u\right|^{2}=0, \quad \text{ on} \quad \mathbb{R}^{+}.
\end{equation}
Then there exist $\theta_{}\in \mathbb{R}$ and $c\in\mathbb{R}$  such that 
\begin{equation}\label{CLO}
u(x)=e^{i\theta_{} }e^{\frac{\omega+1}{2}}e^{-\frac{1}{2}(x+c)^{2}},\quad  \text{ for all $x\in\mathbb{R}^{+}$}.
\end{equation}
\end{lemma} 
\begin{proof}
We  may write  $u(x)=e^{i\theta(x)}\rho (x)$, where $\theta$, $\rho\in C^{2}(\mathbb{R}^{+})$ and $\rho\geq0$. Multiplying the equation \eqref{A22} by $\partial^{}_{x}\overline{u}$, we obtain
\begin{equation}\label{EQQ123}
\left|\partial^{}_{x}u(x)\right|^{2}-\left(\omega+1\right)\left|u(x)\right|^{2}+\left|u(x)\right|^{2}\mathrm{Log}\left|u(x)\right|^{2}\equiv K,
\end{equation}
where $K\in \mathbb{R}$. Since $u\in H^{1}(\mathbb{R}^{+})$, it follows that $u(x)\rightarrow 0$ as $x\rightarrow \infty$. Thus, by \eqref{A22}, $\partial^{2}_{x}u(x)\rightarrow 0$ as $x\rightarrow \infty$, and so $\partial^{}_{x}u(x)\rightarrow 0$ as $x\rightarrow \infty$. Then, letting  $x\rightarrow \infty$  in \eqref{EQQ123}, we see that $K=0$, so
\begin{equation}\label{EQQ0}
\left|\partial^{}_{x}u(x)\right|^{2}-\left(\omega+1\right)\left|u(x)\right|^{2}+\left|u(x)\right|^{2}\mathrm{Log}\left|u(x)\right|^{2}\equiv 0.
\end{equation}

Next, writing of the system of equations satisfied by $\theta$ and $\rho$ we have in particular that $\rho\,\partial^{2}_{x}\theta+ 2\partial^{}_{x}\rho\,\partial^{}_{x}\theta\equiv0$. Which implies that there exists $\hat{K}\in \mathbb{R}$ such that $\rho^{2}\partial^{}_{x}\theta\equiv \hat{K}$. On the other hand, by \eqref{EQQ0} we have that $\left|\partial^{}_{x}u\right|$ is bounded, it follows that $\rho^{2}(\partial^{}_{x}\theta)^{2}$ is bounded. Since $\rho(x)\rightarrow 0$ as $x\rightarrow \infty$, we must have $\hat{K}\equiv0$. Therefore, $\partial^{}_{x}\theta\equiv 0$ and  $u(x)=e^{i\theta}\rho (x)$, where $\theta\in \mathbb{R}$ and  $\rho\geq0$ satisfies
\begin{equation}\label{A33}
-\partial^{2}_{x}\rho+\omega \rho-\rho\,  \mathrm{Log}\left|\rho\right|^{2}=0, \quad  \text{ on} \quad \mathbb{R}^{+}.
\end{equation}
We remark that if we take $\beta(s)=\omega s-s\,\mathrm{Log}s^{2}$, since $\beta\in C[0,+\infty)$, is nondecreasing for $s$ small,  $\beta(0) = 0$ and $\beta(\sqrt{e^{\omega}})=0$, by \cite[Theorem 1]{LVL} we have that each solution $\rho\geq 0$ is either trivial or strictly positive.
Since $u\neq 0$, we infer that $\rho>0$ on $\mathbb{R}^{+}$. Finally, the equation \eqref{A33} may be integrated using standard arguments. Indeed, by explicit integration there exists $c\in \mathbb{R}$ such that for $x>0$ (see \cite{CMKA}),
 \begin{equation*}
\rho(x)=e^{\frac{\omega+1}{2}}e^{-\frac{1}{2}(x+c)^{2}},
\end{equation*}
which completes the proof.
\end{proof}

\begin{proof}[ \bf {Proof of Theorem \ref{ESTIO}}] Let $\varphi$ be a  solution to \eqref{GS}. From Lemma \ref{LCU} and from characterization given by Lemma  \ref{L8}, the restriction $\varphi_{i}:\mathbb{R}^{+}\rightarrow \mathbb{C}$  of the stationary state $\varphi$ must satisfy
\begin{equation}\label{ESTA2}
\varphi_{i}(x)=e^{i\theta_{i}}e^{\frac{\omega+1}{2}}e^{-\frac{1}{2}(x+c_{i})^{2}},
\end{equation}
where $\theta_{i}$, $c_{i}\in\mathbb{R}$. By continuity at the vertex we have that $e^{i\theta_{1}}=\ldots=e^{i\theta_{N}}$ and $c_{i}=\epsilon_{i}h$ with $h>0$ and $\epsilon_{i}=\pm 1$. Moreover, by the jump condition we see that 
\begin{equation*}
h\sum^{N}_{i=1}\epsilon_{i}=\gamma.
\end{equation*}
Now we determine $\epsilon_{i}$ and $h$. As in \cite{NQTG, AQFF}, due to the bell shape of the function $\varphi_{i}$, we say that a stationary state $\varphi$ has a ``bump''(resp. ``tail'') on the edge $i$ if $\epsilon_{i}=-1$ (resp. $\epsilon_{i}=1$). Notice that $\sum^{N}_{i=1}\epsilon_{i}$ must have the same sign of $\gamma$. In particular, a stationary state  must have more tails than bumps. We choose to index the stationary states by the number $\kappa$ of bumps. Therefore, we see that the functions
\begin{align}\label{E5rt}
\left(\varphi^{\kappa}\right)_{i}(x)&=
\begin{cases}
e^{i\theta_{}}e^{\frac{\omega+1}{2}}e^{-\frac{1}{2}(x-h_{{\kappa}})^{2}}, &\text{$i=1$, $\ldots$, $\kappa$} \\
e^{i\theta_{}}e^{\frac{\omega+1}{2}}e^{-\frac{1}{2}(x+h_{{\kappa}})^{2}}, &\text{$i=\kappa+1$, $\ldots$, $N$;} 
\end{cases}
\end{align}
where $\theta\in \mathbb{R}$, $\kappa=0$,  $\ldots$, $\left[(N-1)/2\right]$ and  $h_{{\kappa}}={\gamma}/{(N-2\kappa)}$ are the only candidates to be solution to \eqref{GS}. Conversely we can verify directly that this is indeed the case. This conclude the proof of Theorem \ref{ESTIO}.
\end{proof}

\section{ Minimization problem for the Kirchhoff case}
\label{S:15}
The aim of this section is to prove that the infimum of the action functional for the Kirchhoff case $S_{\omega,0}$ restrict to the Nehari manifold is given by $d_{0}(\omega)=2d_{\mathbb{R}^{+}}(\omega)$ (see \eqref{MPEaame}). The knowledge of  $d_{0}(\omega)$ will be a key ingredient in the next section in the proof of the existence of the ground states for $\gamma>0$.
\begin{proposition} \label{LEM1}
Let $\omega\in \mathbb{R}$. Then $d_{0}(\omega)=2d_{\mathbb{R}^{+}}(\omega)$.
\end{proposition}
Before proceeding  to the proof of Proposition \ref{LEM1}, we establish the following lemma. 
\begin{lemma} \label{LEM2}
Let $\omega\in \mathbb{R}$. Then there exists a sequence of  functions $\left\{\varphi_{n}\right\}_{n\in \mathbb{N}}\subseteq W(\mathbb{R}^{+})$ with $\varphi_{n}(0)=0$ such that 
\begin{equation*}
\lim_{n \to +\infty}\|\varphi_{n}\|^{2}_{L^{2}(\mathbb{R}^{+})}=4d_{\mathbb{R}^{+}}(\omega)\quad \text{and }\quad \lim_{n \to +\infty}I_{\mathbb{R}^{+}}(\varphi_{n},\omega)=0.
\end{equation*}
\end{lemma}
\begin{proof}
Let $n\in \mathbb{N}$. We consider the sequence of functions 
\begin{align}\label{ttrt}
\varphi_{n}(x)&=
\text{erf}(x)\phi_{\omega}(x-n) \quad\text{for all $x\in \mathbb{R}^{+}$;} 
\end{align}
where the  functions  $\text{erf}$ and $\phi_{\omega}$ are defined by \eqref{erf} and  \eqref{AUXF} respectively. It is clear that $\varphi_{n}\in W(\mathbb{R}^{+})$ with $\varphi_{n}(0)=0$ for every  $n\in \mathbb{N}$. Set $\hat{\psi}_{n}(x)=\phi_{\omega}(x-n)$ for all $x>0$. We claim that $\varphi_{n}\rightarrow \hat{\psi}_{n}$ strongly in $W(\mathbb{R}^{+})$ as $n\rightarrow+\infty$. Indeed, by elementary computations we see that
\begin{equation*}
\|\varphi_{n}-\hat{\psi}_{n}\|^{2}_{L^{2}(\mathbb{R}^{+})}\leq \sqrt{2\pi}e^{\frac{\omega+1}{2}}e^{-\frac{n^{2}}{2}}\quad  \text{and}\quad  \|\partial_{x}\varphi_{n}-\partial_{x}\hat{\psi}_{n}\|^{2}_{L^{2}(\mathbb{R}^{+})}\leq  4(n^{2}+1)e^{\frac{\omega+1}{2}}e^{-\frac{n^{2}}{2}},
\end{equation*}
which implies that $\varphi_{n}\rightarrow \hat{\psi}_{n}$ strongly in $H^{1}(\mathbb{R}^{+})$ as $n\rightarrow+\infty$. On the other hand, we remark that $|\varphi_{n}(x)-\hat{\psi}_{n}(x)|\leq e^{\frac{\omega+1}{2}}e^{-n}$ for every $x\in \mathbb{R}^{+}$ and $n\geq3$. Thus, by the definition of $A(s)$ given in \eqref{IFD}, we have that for sufficiently large $n$,
\begin{align*}
\int_{\mathbb{R}^{+}}A(|\varphi_{n}-\hat{\psi}_{n}|)dx&=-\int_{\mathbb{R}^{+}}|\varphi_{n}-\hat{\psi}_{n}|^{2}\mbox{Log}|\varphi_{n}-\hat{\psi}_{n}|^{2}dx \\
&\leq  \int_{\mathbb{R}^{+}}|\varphi_{n}-\hat{\psi}_{n}|dx+ \int_{\mathbb{R}^{+}}|\varphi_{n}-\hat{\psi}_{n}|^{2}dx\leq 3e^{-\frac{n^{2}}{3}},
\end{align*}
which immediately induces, by \eqref{DA1}, that $\varphi_{n}\rightarrow \hat{\psi}_{n}$ strongly in $L^{A}(\mathbb{R}^{+})$ as $n\rightarrow+\infty$. In particular,  $\|\varphi_{n}-\hat{\psi}_{n}\|^{2}_{W(\mathbb{R}^{+})}\rightarrow 0$. To conclude, we remark that  by continuity 
\begin{align*}
\lim_{n \to +\infty}\|\varphi_{n}\|^{2}_{L^{2}(\mathbb{R}^{+})}&=\lim_{n \to +\infty}\|\hat{\psi}_{n}\|^{2}_{L^{2}(\mathbb{R}^{+})}=e^{\omega+1}\sqrt{\pi}\lim_{n \to +\infty}\frac{\text{erf}(n)+1}{2}=e^{\omega+1}\sqrt{\pi} \\
\lim_{n \to +\infty}I_{\mathbb{R}^{+}}(\varphi_{n},\omega)&=\lim_{n \to +\infty}I_{\mathbb{R}^{+}}(\hat{\psi}_{n},\omega)=-e^{\omega+1}\lim_{n \to +\infty}ne^{-n^{2}}=0,  \\
\end{align*}
and the lemma is proved.
\end{proof}

\begin{proof}[ \bf {Proof of Proposition \ref{LEM1}}] We use the argument in \cite[Theorem 3]{AQFF}.
First, we claim that $d_{0}(\omega)\geq 2d_{\mathbb{R}^{+}}(\omega)$. Indeed, let $u\in W(\Gamma)$ such that $I_{\omega,\gamma}(u)=0$ and let $u^{\ast}$ be its symmetric rearrangement. Then, using Proposition \ref{RSS} we see that $u^{\ast}$ is positive, symmetric and $u^{\ast}\in W(\Gamma)$. Moreover, by property (ii) of Proposition \ref{RSS} we have that
\begin{equation*}
Q_{\omega,\gamma}(u^{\ast}):=\frac{4}{N^{2}}\|\partial_{x}u^{\ast}\|^{2}_{L^{2}(\Gamma)}+\omega\|u^{\ast}\|^{2}_{L^{2}(\Gamma)}-\int_{\Gamma}\left|u^{\ast}\right|^{2}\mathrm{Log}\left|u^{\ast}\right|^{2}dx\leq I_{\omega,\gamma}(u)=0,
\end{equation*}
which combined with $2S_{\omega,\gamma}(u)=I_{\omega,\gamma}(u)+\|u\|^{2}_{L^{2}(\Gamma)}$ and the properties  of $u^{\ast}$  implies that
\begin{align}\notag
d_{0}(\omega)&=\frac{1}{2}\,{\inf}\bigl\{\left\|u\right\|_{L^{2}(\Gamma)}^{2}:u\in  W(\Gamma) \setminus \left\{0 \right\},  I_{\omega,\gamma}(u)= 0 \bigl\}\\ 
&\geq\frac{1}{2}\,{\inf}\bigl\{\|u\|_{L^{2}(\Gamma)}^{2}:u\in  W(\Gamma) \setminus \left\{0 \right\},\, \text{$u$ symmetric},\,  Q_{\omega,\gamma}(u)\leq 0 \bigl\}. \label{mn}
\end{align}
Next we use the scaling $u_{\lambda}(\cdot):=\lambda^{1/2}u(\lambda\cdot)$ with $\lambda=N/2$.  From  a simple calculation, we obtain
\begin{equation}\label{Po}
Q_{\omega,\gamma}(u_{\lambda})=I_{\omega+\text{Log}\left(2/N\right),\gamma}(u).
\end{equation}
Thus, combining \eqref{mn}, \eqref{Po} and due to the symmetry of $u$ leads to
\begin{align*}
d_{0}(\omega)&\geq\frac{1}{2}{\inf}\bigl\{\left\|u\right\|_{L^{2}(\Gamma)}^{2}:u\in  W(\Gamma) \setminus \left\{0 \right\},\, \text{$u$ symmetric},\,  I_{\omega+\text{Log}\left(2/N\right),\gamma}(u)\leq 0\bigl\}\\
&=\frac{N}{2}{\inf}\left\{\left\|v\right\|_{L^{2}(\mathbb{R}^{+})}^{2}:\, v\in W(\mathbb{R}^{+})\setminus  \left\{0 \right\},  I_{\mathbb{R}^{+}}(v,\omega+\text{Log}\left(2/N\right))\leq0\right\}\\
&=N d_{\mathbb{R}^{+}}(\omega+\text{Log}\left(2/N\right))=2d_{\mathbb{R}^{+}}(\omega).
\end{align*}
Secondly, we prove that the lower bound $2d_{\mathbb{R}^{+}}(\omega)$ is optimal by means of a minimizing sequence.  Let $n\in \mathbb{N}$. We consider the sequence of functions 
\begin{equation*}
(u_{n})_{i}(x)=
\begin{cases}
\varphi_{n}(x) &\text{if $i=1$;}\\
0   &\text{if $ i\neq1$;}
\end{cases}
\end{equation*}
where the  function $\varphi_{n}$ is defined by \eqref{ttrt}. Notice that the sequence $u_{n}$ belongs to $W(\Gamma)$. Moreover, it follows from Lemma \ref{LEM2} that 
\begin{equation}\label{wer}
\lim_{n\rightarrow \infty} I_{\omega,0}(u_{n})=\lim_{n\rightarrow \infty}I_{\mathbb{R}^{+}}(\varphi_{n},\omega)=0. 
\end{equation}
Define the sequence $v_{n}(x)=\lambda_{n}u_{n}(x)$ with 
\begin{equation*}
\lambda_{n}=\exp\left(\frac{I_{\omega,0}(u_{n})}{2\|u_{n}\|^{2}_{L^{2}(\Gamma)}}\right),
\end{equation*}
where $\exp(x)$ represent the exponential function. Then, it follows from  \eqref{wer} that $\lim_{n\rightarrow \infty}\lambda_{n}=1$. Moreover,  an easy calculation shows that $I_{\omega,0}(v_{n})=0$ for any $n\in \mathbb{N}$. Thus,  by the definition of $d_{0}(\omega)$ and Lemma \ref{LEM2} leads to
\begin{equation*}
d_{0}(\omega)\leq \lim_{n\rightarrow \infty} S_{\omega,0}(\lambda_{n}u_{n})=\frac{1}{2}\lim_{n\rightarrow \infty}\,\lambda_{n}^{2}\|u_{n}\|^{2}_{L^{2}(\Gamma)}=2d_{\mathbb{R}^{+}}(\omega),
\end{equation*}
and the proposition is proved.
\end{proof}

\section{Existence and identification of the ground state}
\label{S:3}
Before giving the proof of Theorem  \ref{ESSW}, we need to establish some preliminaries. Firstly we extend the  one-dimensional logarithmic Sobolev inequality to star graphs. 
\begin{lemma} \label{L1}
Let $u$ be any function in $H^{1}(\Gamma)$ and $\alpha$ be any positive number. Then
\begin{equation}\label{LSA}
\int_{\Gamma}\left|u\right|^{2}\mathrm{Log}\left|u\right|^{2}dx\leq \frac{\alpha^{2}}{\pi} \|\partial_{x}u\|^{2}_{L^{2}(\Gamma)}+\left(\mathrm{Log} \left(2\|u \|^{2}_{L^{2}(\Gamma)}\right)-\left(1+\mathrm{Log}\,\alpha\right)\right)\|u \|^{2}_{L^{2}(\Gamma)}.
\end{equation}
\end{lemma}
\begin{proof}
The proof of \eqref{LSA}  follow immediately from the standard  logarithmic Sobolev inequality on $H^{1}(\mathbb{R})$ (see \cite[Theorem 8.14]{ELL}), considering  that any function in $H^{1}(\mathbb{R}^{+})$ can be extended to an even function in $H^{1}(\mathbb{R})$, and applying this reasoning to each component of $u$. We omit the details.
\end{proof}

\begin{lemma}\label{L2}
Let $\gamma>0$ and $\omega\in \mathbb{R}$. Then, the quantity $d_{\gamma}(\omega)$ is positive and satisfies
\begin{equation}\label{EA}
d_{\gamma}(\omega)\geq  \frac{1}{4}\sqrt{\frac{\pi}{2}} e^{\omega+1}e^{-2\frac{\gamma^{2}}{N^{2}}}.
\end{equation}
\end{lemma}
\begin{proof}
Notice that if  $f\in H^{1}(\mathbb{R}^{+})$, then  by H\"{o}lder and Sobolev inequalities we have that
\begin{align}
\left|f(0)\right|^{2}\leq \epsilon \|f_{}\|^{2}_{L^{2}(\mathbb{R}^{+})}+ \epsilon^{-1} \|\partial_{x}f\|^{2}_{L^{2}(\mathbb{R}^{+})} .\label{EA2}
\end{align}
Now, let $u\in W(\Gamma) \setminus  \left\{0 \right\}$ be such that  $I_{\omega,\gamma}(u)=0$. From \eqref{EA2} with $\epsilon=2\gamma/N$ we see that
\begin{align}
\gamma\left|u_{1}(0)\right|^{2}\leq& \frac{\gamma}{N} \sum^{N}_{i=1} \left\{\epsilon \|u_{i}\|^{2}_{L^{2}(\mathbb{R}^{+})}+ \epsilon^{-1} \|\partial_{x}u_{i}\|^{2}_{L^{2}(\mathbb{R}^{+})}\right\}\nonumber\\ \label{HJK}
&=\frac{2\gamma^{2}}{N^{2}} \|u_{}\|^{2}_{L^{2}(\Gamma)}+ \frac{1}{2}\|\partial_{x}u\|^{2}_{L^{2}(\Gamma)},
\end{align}
which combined with $I_{\omega,\gamma}(u)=0$ and the logarithmic Sobolev inequality \eqref{LSA} with $\alpha=\sqrt{\frac{\pi}{2}}$ gives
\begin{equation*}
 \left(\omega+1+\mbox{Log}\left(\sqrt{\frac{\pi}{2}}\right)-\frac{2\gamma^{2}}{N^{2}}\right)\left\|u\right\|^{2}_{L^{2}(\Gamma)}\leq \left(\mbox{Log}\left(2\left\|u\right\|^{2}_{L^{2}(\Gamma)}\right) \right)\left\|u\right\|^{2}_{L^{2}(\Gamma)}.
\end{equation*}
Thus, $\left\|u\right\|^{2}_{L^{2}(\Gamma)}\geq \frac{1}{2}\sqrt{\frac{\pi}{2}} e^{\omega+1}e^{-\frac{2\gamma^{2}}{N^{2}}}$. Finally, by the definition of $d_{\gamma}(\omega)$ given in \eqref{MPE}, we get \eqref{EA}.
\end{proof}

\begin{lemma} \label{L34}
Let $N\geq2$, $\omega \in \mathbb{R}$ and  $\gamma^{\ast}(N):=N\left(\mathrm{erf}^{-1}(1-2/N)\right)$. If $\gamma>\gamma^{\ast}(N)$, then the following  inequality holds:
\begin{equation}\label{EIN}
 d_{\gamma}(\omega)<d_{0}(\omega).
\end{equation} 
\end{lemma}
\begin{proof}
We consider the symmetric function  $\phi^{0}_{\omega,\gamma}$  defined by \eqref{E5}. Then, it is clear that $\phi^{0}_{\omega,\gamma}\in W(\Gamma)$ and ${I}_{\omega,\gamma}(\phi^{0}_{\omega,\gamma})=0$. Moreover, by elementary computations we see that
\begin{equation}\label{nbm}
{S}_{\omega,\gamma}(\phi^{0}_{\omega,\gamma})=\frac{N}{4}\sqrt{\pi}e^{\omega+1}\left(1-\text{erf}\left(\frac{\gamma}{N}\right)\right)=\frac{N}{2}\left(1-\text{erf}\left(\frac{\gamma}{N}\right)\right)d_{0}(\omega).
\end{equation}
Since $\gamma>\gamma^{\ast}(N)$, this implies 
\begin{equation*}
\frac{N}{2}\left(1-\text{erf}\left(\frac{\gamma}{N}\right)\right)<1,
\end{equation*}
which combined with \eqref{nbm} and by the definition of ${d}_{\gamma}(\omega)$ gives
\begin{equation*}
{d}_{\gamma}(\omega)\leq {S}_{\omega,\gamma}(\phi^{0}_{\omega,\gamma})<d_{0}(\omega),
\end{equation*}
and the lemma is proved.
\end{proof}
The proof of the following lemma can be found in \cite[Lemma 4.10]{AHAA}.
\begin{lemma} \label{L4}
Let  $\left\{u_{n}\right\}$ be a bounded sequence in $W({\mathbb{R}^{+}})$ such that $u_{n}\rightarrow u$ a.e. in $\mathbb{R}^{+}$. Then $u\in W({\mathbb{R}}^{+})$ and 
\begin{equation*}
\lim_{n\rightarrow \infty}\int_{\mathbb{R}^{+}}\left\{\left|u_{n}\right|^{2}\mathrm{Log}\left|u_{n}\right|^{2}dx-\left|u_{n}-u\right|^{2}\mathrm{Log}\left|u_{n}-u\right|^{2}\right\}dx=\int_{\mathbb{R}^{+}}\left|u\right|^{2}\mathrm{Log}\left|u\right|^{2}dx.
\end{equation*}
\end{lemma}

\begin{proof}[ \bf {Proof of Theorem \ref{ESSW}}] 
First, every minimizing sequence of \eqref{MPE} is bounded in $ W(\Gamma)$. Let $\left\{ u_{n}\right\}$ be a minimizing sequence. We remark that the sequence $\left\{ u_{n}\right\}$ is bounded in $L^{2}(\Gamma)$. Now, by \eqref{HJK}, the logarithmic Sobolev inequality \eqref{LSA} and recalling that $I_{\omega,\gamma}(u_{n})=0$, we see for $\alpha>0$,
\begin{equation*}
\left(\frac{1}{2}-\frac{\alpha^{2}}{\pi}\right)\left\|\partial_{x}u_{n}\right\|^{2}_{L^{2}(\Gamma)}\leq \Big(\mbox{Log}\bigg(\frac{e^{\frac{2\gamma^{2}}{N^{2}}}e^{-\left(\omega+1\right)}}{\alpha^{}}\bigg)\Big)\left\|u_{n}\right\|^{2}_{L^{2}(\Gamma)} +\mbox{Log}\left(2\left\|u_{n}\right\|^{2}_{L^{2}(\Gamma)}\right) \left\|u_{n}\right\|^{2}_{L^{2}(\Gamma)}.
\end{equation*}
Taking $\alpha>0$ sufficiently small, we have that $\left\{ u_{n}\right\}$ is bounded in $H^{1}(\Gamma)$. Moreover,  it follows from  $I_{\omega,\gamma}(u_{n})=0$, \eqref{HJK} and \eqref{DB} that
\begin{equation*}
\frac{1}{2}\left\|\partial_{x}u_{n}\right\|^{2}_{L^{2}(\Gamma)}+\int_{\Gamma}A\left(\left|u_{n}\right|\right)dx\leq  C,
\end{equation*}
which implies, by \eqref{DA1}, that the sequence $\left\{ u_{n}\right\}$ is bounded in $W(\Gamma)$. In addition, since $W^{}(\Gamma)$ is a reflexive Banach space, there exists $\varphi \in W(\Gamma)$ such that, up to a subsequence, $u_{n}\rightharpoonup \varphi$ weakly in $W^{}(\Gamma)$. Moreover, as it was observed in the proof of \cite[Theorem 1]{AQFF}, by weak convergence we have $u_{n}(0)\rightarrow\varphi(0)$.

Now we show that $\varphi$ is nontrivial. Suppose that $\varphi\equiv 0$. Since $u_{n}$ satisfies ${I}_{\omega,\gamma}(u_{n})=0$,  we obtain
\begin{equation}\label{AA}
\lim_{n \to \infty} I_{\omega,0}(u_{n})=\gamma \lim_{n \to \infty} \left|u_{1,n}(0)\right|^{2}= 0.
\end{equation}
Define the sequence $v_{n}=\lambda_{n}u_{n}$ with 
\begin{equation*}
\lambda_{n}=\exp\left(\frac{I_{\omega,0}(u_{n})}{2\|u_{n}\|^{2}_{L^{2}(\Gamma)}}\right),
\end{equation*}
where $\exp(x)$ represents the exponential function. Then, it follows from  \eqref{AA} that $\lim_{n\rightarrow \infty}\lambda_{n}=1$. Moreover,  an easy calculation shows that $I_{\omega,0}(v_{n})=0$ for any $n\in \mathbb{N}$.  Thus, by  the definition of $d_{\gamma}(\omega)$, we see that
\begin{equation*}
d_{0}(\omega)\leq \lim_{n \to \infty}S_{\omega,0}(v_{n})= \frac{1}{2} \lim_{n \to \infty}\left\{\lambda^{2}_{n} \|u_{n}\|^{2}_{L^{2}(\Gamma)}\right\}= d_{\gamma}(\omega),
\end{equation*}
that it is contrary to \eqref{EIN} and therefore we conclude that $\varphi$ is nontrivial.

Secondly, we prove that  $I_{\omega,\gamma}(\varphi)=0$ and $\varphi\in \mathcal{G}_{\omega,\gamma}$.  If we suppose that $I_{\omega,\gamma}(\varphi)<0$, by elementary computations we find that  there is $\lambda\in (0,1)$ such that $I_{\omega,\gamma}(\lambda \varphi)=0$. Then, from the definition of $d_{\gamma}(\omega)$ and  the weak lower semicontinuity of the $L^{2}(\Gamma)$-norm, we have
\begin{equation*}
d_{\gamma}(\omega)\leq \frac{1}{2}\left\|\lambda \varphi\right\|^{2}_{L^{2}(\Gamma)}<\frac{1}{2}\left\|\varphi\right\|^{2}_{L^{2}(\Gamma)}\leq \frac{1}{2}\liminf\limits_{n\rightarrow \infty}\left\|u_{n}\right\|^{2}_{L^{2}(\Gamma)}=d_{\gamma}(\omega),
\end{equation*}
which is impossible. On the other hand, assume that $I_{\omega,\gamma}(\varphi)>0$. Since the embedding  $W(\Gamma)\hookrightarrow {H^{1}(\Gamma)}$ is continuous, we see that $u_{n}\rightharpoonup \varphi$ weakly in $H^{1}(\Gamma)$. Thus, we have  
\begin{align}
& \left\|u_{n}\right\|^{2}_{L^{2}(\Gamma)}-\left\|u_{n}-\varphi\right\|^{2}_{L^{2}(\Gamma)}-\left\|\varphi\right\|^{2}_{L^{2}(\Gamma)}\rightarrow0 \label{2C11}\\
&\left\|\partial_{x}u_{n}\right\|^{2}_{L^{2}(\Gamma)}-\left\|\partial_{x}u_{n}-\partial_{x}\varphi\right\|^{2}_{L^{2}(\Gamma)}-\left\|\partial_{x}\varphi \right\|^{2}_{L^{2}(\Gamma)}\rightarrow0, \label{2C12} 
\end{align}
as $n\rightarrow\infty$. Combining \eqref{2C11}, \eqref{2C12} and Lemma \ref{L4} leads to 
\begin{equation*}
\lim_{n\rightarrow \infty}I_{\omega, \gamma}(u_{n}-\varphi)=\lim_{n\rightarrow \infty}I_{\omega, \gamma}(u_{n})-I_{\omega,\gamma}(\varphi)=-I_{\omega,\gamma}(\varphi),
\end{equation*}
which combined with  $I_{\omega, \gamma}(\varphi)> 0$ give us  that $I_{\omega, \gamma}(u_{n}-\varphi)<0$ for sufficiently large $n$. Thus, by \eqref{2C11} and  applying the same argument as above, we see that 
\begin{equation*}
d_{\gamma}(\omega)\leq\frac{1}{2} \lim_{n\rightarrow \infty}\left\|u_{n}-\varphi\right\|^{2}_{L^{2}(\Gamma)}=d_{\gamma}(\omega)-\frac{1}{2}\left\|\varphi\right\|^{2}_{L^{2}(\Gamma)},
\end{equation*}
which is a contradiction because $\|\varphi\|^{2}_{L^{2}(\Gamma)}>0$. Therefore, we deduce that $I_{\omega,\gamma}(\varphi)=0$. To conclude,  by the weak lower semicontinuity of the $L^{2}(\Gamma)$-norm,  we have
\begin{equation}\label{inequa}
d_{\gamma}(\omega)\leq \frac{1}{2}\left\|\varphi\right\|^{2}_{L^{2}(\Gamma)}\leq \frac{1}{2} \liminf\limits_{n\rightarrow \infty}\left\|u_{n}\right\|^{2}_{L^{2}(\Gamma)}=d_{\gamma}(\omega),
\end{equation}
which implies, by the definition of $d_{\gamma}(\omega)$, that $\varphi\in \mathcal{G}_{\omega,\gamma}$.

Finally, we prove that $\phi^{0}_{\omega,\gamma}$ is the ground state. By Remark \ref{RULT} and Theorem \ref{ESTIO},  it is sufficient to verify that 
\begin{equation}\label{PRG1}
S_{\omega,\gamma}(\phi^{\kappa}_{\omega,\gamma})<S_{\omega,\gamma}(\phi^{\kappa+1}_{\omega,\gamma})\quad \text{for $0\leq \kappa\leq\left[(N-1)/2\right]-1$.}
\end{equation}
For $\gamma>0$ and $x\,\in \mathbb{R}^{+}$ we set $f_{\gamma}(x)=x(\text{erf}(\gamma/x))$,  where the  function  $\text{erf}$ is defined by \eqref{erf}. By elementary computations we have that
\begin{equation}\label{123}
S_{\omega,\gamma}(\phi^{\kappa}_{\omega,\gamma})=\frac{\sqrt{\pi}}{4}e^{\omega+1}\left(N-f_{\gamma}(N-2\kappa)\right).
\end{equation}
We claim that $f_{\gamma}$  is strictly increasing on $\mathbb{R}^{+}$. Indeed, it is clear that $f_{\gamma}\in C^{2}(\mathbb{R}^{+})$, $f_{\gamma}(0)=0$ and $f_{\gamma}(x)\rightarrow 2\gamma/\sqrt{\pi}$ as $x\rightarrow\infty$. Moreover, we have that $f^{\prime\prime}_{\gamma}(x)<0$ for all $x\in \mathbb{R}^{+}$, which combined with 
\begin{equation*}
\lim_{x\rightarrow 0^{+}}f^{\prime}_{\gamma}(x)=\gamma, \quad\quad  \lim_{x\rightarrow \infty}f^{\prime}_{\gamma}(x)=0
\end{equation*}
implies that $f^{\prime}_{\gamma}(x)>0$ for all $x\in \mathbb{R}^{+}$. In particular, 
\begin{equation}\label{567}
f_{\gamma}(N-2(\kappa+1))<f_{\gamma}(N-2\kappa).
\end{equation}
Combining \eqref{123} and \eqref{567} we get \eqref{PRG1}. This completes the proof of theorem. 
\end{proof}

\section{Stability of the ground states}
\label{S:4}
This section is devoted to the proof of Theorem \ref{EST}. We first prove compactness of the minimizing sequences.
\begin{lemma} \label{CSM}
Let $N\geq2$, $\omega \in \mathbb{R}$ and   $\gamma>\gamma^{\ast}(N)=N\left(\mathrm{erf}^{-1}(1-2/N)\right)$. Let $\left\{ u_{n}\right\}\subseteq W(\Gamma)$ be a minimizing sequence for $d_{\gamma}(\omega)$. Then, up to a subsequence,  there is  $\theta_{}\in \mathbb{R}$ such that $u_{n}\rightarrow e^{i\theta_{}}\phi^{0}_{\omega, \gamma}$ in $W(\Gamma)$.
\end{lemma}
\begin{proof}
We see by the proof of Theorem \ref{ESSW} that there is  $\varphi\in \mathcal{G}_{\omega,\gamma}$ such that, up to a subsequence, $u_{n}\rightharpoonup \varphi$ weakly  in $W(\Gamma)$ and $u_{i,n}(x)\rightarrow\varphi_{i}(x)$ $a.e.$ in $\mathbb{R}^{+}$ for $i=1$, $\ldots$, $N$. Furthermore, by  \eqref{2C11} and \eqref{inequa} we have  $u_{n}\rightarrow \varphi$   in $L^{2}(\Gamma)$. Then, since the sequence $\left\{ u_{n}\right\}$ is bounded in $H^{1}(\Gamma)$, from \eqref{DB} we obtain
\begin{equation*}
 \lim_{n\rightarrow \infty}\int_{\Gamma}B\left(\left|u_{n}\right|\right)dx=\int_{\Gamma}B\left(\left|\varphi_{}\right|\right)dx.
\end{equation*}
Thus,  since $I_{\omega,\gamma}(u_{n})=I_{\omega,\gamma}(\varphi)=0$  for any $n\in \mathbb{N}$,  we obtain
\begin{equation}\label{2BX1}
\lim_{n\rightarrow \infty}\left\{\left\|\partial_{x}u_{n}\right\|^{2}_{L^{2}(\Gamma)}+\int_{\Gamma}A\left(\left|u_{n}\right|\right)dx\right\}=\left\{\left\|\partial_{x}\varphi\right\|^{2}_{L^{2}(\Gamma)}+\int_{\Gamma}A\left(\left|\varphi\right|\right)dx\right\}.
\end{equation}
Moreover, by weak lower semi-continuity of the $L^{2}$-norm and Fatou lemma, we deduce 
\begin{align}
& \left\|\partial_{x}\varphi_{i}\right\|^{2}_{L^{2}(\mathbb{R}^{+})}\leq\liminf_{n\rightarrow \infty}\left\|\partial_{x}u_{i,n}\right\|^{2}_{L^{2}(\mathbb{R}^{+})}, \label{LL1} \\
&\int_{\mathbb{R}^{+}}A\left(\left|\varphi_{i}(x)\right|\right)dx\leq \liminf_{n \to \infty}\int_{\mathbb{R}^{+}}A\left(\left|u_{i,n}(x)\right|\right)dx,
\label{LL2} \end{align}
for every $i=1$, $\ldots$, $N$. Therefore, by \eqref{2BX1}, \eqref{LL1} and \eqref{LL2} we infer that (see, for example, \cite[Lemma 12 in chapter V]{AH})
\begin{align}
& \lim_{n\rightarrow \infty}\left\|\partial_{x}u_{i,n}\right\|^{2}_{L^{2}(\mathbb{R}^{+})}=\left\|\partial_{x}\varphi_{i}\right\|^{2}_{L^{2}(\mathbb{R}^{+})},\label{N1}\\
& \lim_{n \to \infty}\int_{\mathbb{R}^{+}}A\left(\left|u_{i,n}(x)\right|\right)dx=\int_{\mathbb{R}^{+}}A\left(\left|\varphi_{i}(x)\right|\right)dx. \label{N2} \end{align}
Since $u_{n}\rightharpoonup \varphi$ weakly  in $H^{1}(\Gamma)$, it follows from \eqref{N1} that $u_{i,n}\rightarrow\varphi_{i}$  in $H^{1}(\mathbb{R}^{+})$.  Furthermore, by Proposition \ref{orlicz}-{ii)} and \eqref{N2} we have $u_{i,n}\rightarrow\varphi_{i}$  in $L^{A}(\mathbb{R}^{+})$. Thus, by definition of the $W(\Gamma)$-norm, we infer that $u_{n}\rightarrow\varphi$  in $W(\Gamma)$  and the conclusions follow directly from Theorem  \ref{ESSW}.
\end{proof}

\begin{proof}[ {\bf{Proof of Theorem \ref{EST}}}] 
We argue by contradiction.  Suppose that $\phi^{0}_{\omega, \gamma}$ is not stable in $W(\Gamma)$. Then there exist $\epsilon>0$, a sequence $(u_{n,0})_{n\in \mathbb{N}}$ such that
\begin{equation}\label{3C1}
\left\|u_{n,0}-\phi^{0}_{\omega, \gamma}\right\|_{ W(\Gamma)}<\frac{1}{n},
\end{equation}
and a sequence $(\tau_{n})_{n\in \mathbb{N}}$ such that
\begin{equation}\label{3C2}
{\rm\inf\limits_{\theta\in \mathbb{R}}} \|u_{n}(\tau_{n})-e^{i\theta}\phi^{0}_{\omega, \gamma}\|_{W(\Gamma)}=\frac{\epsilon}{2},
\end{equation}
where $u_{n}$ denotes the solution of the Cauchy problem \eqref{00NL} with initial data $u_{n,0}$.  Set $v_{n}= u_{n}(t_{n})$.  By \eqref{3C1} and conservation laws, as $n\rightarrow \infty$,
\begin{gather}
\left\|v_{n}\right\|^{2}_{L^{2}(\Gamma)}=\left\|u_{n}(t_{n})\right\|^{2}_{L^{2}(\Gamma)}=\left\|u_{n,0}\right\|^{2}_{L^{2}(\Gamma)}\rightarrow \left\|\phi^{0}_{\omega,\gamma}\right\|^{2}_{L^{2}(\Gamma)}\label{CE1} \\
E(v_{n})=E(u_{n}(t_{n}))=E(u_{n,0})\rightarrow E(\phi^{0}_{\omega,\gamma}).
\label{CE2}
\end{gather}
In particular, it follows from \eqref{CE1} and \eqref{CE2} that, as $n\rightarrow \infty$, 
\begin{equation}\label{A12}
S_{\omega,\gamma}(v_{n})\rightarrow S_{\omega,\gamma}(\phi^{0}_{\omega,\gamma})=d_{\gamma}(\omega).
\end{equation}
Moreover, combining \eqref{CE1} and \eqref{A12} leads to $I_{\omega,\gamma}(v_{n})\rightarrow 0$ as $n\rightarrow \infty$.
Define the sequence $f_{n}=\rho_{n}v_{n}$ with
\begin{equation*}
\rho_{n}=\exp\left(\frac{I_{\omega,\gamma}(v_{n})}{2\|v_{n}\|^{2}_{L^{2}(\Gamma)}}\right),
\end{equation*}
where $\exp(x)$ is the exponential function. It is clear that $\lim_{n\rightarrow \infty}\rho_{n}=1$ and $I_{\omega,\gamma}(f_{n})=0$ for any $n\in\mathbb{N}$. Furthermore, since the sequence $\left\{v_{n}\right\}$  is bounded in $W(\Gamma)$, we get $\|v_{n}-f_{n}\|_{W(\Gamma)}\rightarrow 0$ as $n\rightarrow \infty$. Then, by  \eqref{A12}, we have that $\left\{f_{n}\right\}$ is a minimizing sequence for $d_{\gamma}(\omega)$. Thus, by Lemma \ref{CSM}, up to a subsequence, there is $\theta_{0}\in \mathbb{R}$  such that $f_{n}\rightarrow e^{i\theta_{0}}\phi^{0}_{\omega, \gamma}$ in  $W(\Gamma)$. Therefore, by using the triangular inequality, we have
\begin{equation*}
\|u_{n}(t_{n})-e^{i\theta_{0}}\phi^{0}_{\omega,\gamma}\|_{ W(\Gamma)}\leq \|v_{n}-f_{n}\|_{ W(\Gamma)}+\|f_{n}-e^{i\theta_{0}}\phi^{0}_{\omega, \gamma}\|_{ W(\Gamma)}\rightarrow 0,
\end{equation*}
as $n\rightarrow \infty$, it which is a contradiction with \eqref{3C2}. This finishes the proof.
\end{proof}

\section*{Acknowledgements}
The author wishes to express his sincere thanks to the referees for their valuable comments.  The author gratefully acknowledges the support from CNPq, through grant No. 152672/2016-8.

\section{Appendix}
\label{S:5}
The purpose of this Appendix is to show that the energy functional $E$ is of class $C^{1}$ on $W(\Gamma)$.

\begin{proposition} \label{DFFE}
The operator $E: W(\Gamma)\rightarrow \mathbb R$  is of class $C^{1}$ and  for $u\in W(\Gamma)$ the  Fr\'echet derivative of $E$ in $u$ exists and it is given by  
\begin{equation*}
E^{\prime}(u)=-\Delta_{\gamma}u-u\, \mathrm{Log}\left|u\right|^{2}-u.
\end{equation*}
\end{proposition}

The proof of Proposition \ref{DFFE} relies on the following result.

\begin{lemma} \label{APEX23}
The operator $L: u\rightarrow -\Delta_{\gamma}u+u\,  \mathrm{Log}\left|u\right|^{2}$ is continuous from  $W(\Gamma)$  to $W^{\prime}(\Gamma)$.  Moreover, the image under $L$ of a bounded subset of $W(\Gamma)$ is a bounded subset of $W^{\prime}(\Gamma)$.
\end{lemma}
\begin{proof}
 As usual, the operator $-\Delta_{\gamma}$ is naturally extended to $-\Delta_{\gamma}: H^{1}(\Gamma)\rightarrow H^{-1}(\Gamma)$ defined by 
\begin{equation*}
\left\langle -\Delta_{\gamma}u,v\right\rangle=\mathfrak{F_{\gamma}}[u,v],  \quad  \textrm{for} \quad u,v\in H^{1}(\Gamma).
\end{equation*}
Now, using that ${W}(\Gamma)\hookrightarrow H^{1}(\Gamma)$ is a dense embedding, we obtain that $u\rightarrow -\Delta_{\gamma}u$ is continuous from ${{W}}(\Gamma)$ to ${W}^{\prime}(\Gamma)$. On the other hand, by \cite[Lemma 2.3]{CL}, $u\rightarrow u\,\mathrm{Log}\left|u\right|^{2}$ is continuous and bounded from $W(\mathbb{R}^{+})$ to ${W^{\prime}}(\mathbb{R}^{+})$. This implies that $u\rightarrow u\,\mathrm{Log}\left|u\right|^{2}$ is continuous and bounded from $W(\Gamma)$ to ${W^{\prime}}(\Gamma)$, and the lemma  is proved.
\end{proof}

\begin{proof}[ {\bf{Proof of Proposition \ref{DFFE}}}]
We first show that $E$  is continuous. Notice that 
\begin{equation}\label{CCC}
E(u)=\frac{1}{2}\mathfrak{F}_{\gamma}[u]+\frac{1}{2}\int_{\Gamma}A(\left|u\right|)dx-\frac{1}{2}\int_{\Gamma}B(\left|u_{}\right|)dx.
\end{equation}
The first term in the right-hand side of \eqref{CCC} is continuous $H^{1}(\Gamma)\rightarrow \mathbb R$, and it follows from 
Proposition \ref{orlicz}(i) that the second term is continuous $W(\Gamma)\rightarrow \mathbb{R}$. Moreover, by \eqref{DB} we get that the third term  right-hand side of \eqref{CCC} is continuous $H^{1}(\Gamma)\rightarrow \mathbb R$. Therefore, $E\in C(W(\Gamma),\mathbb{R})$. Now, direct calculations show that, for $u$, $v\in W(\Gamma)$, $t\in (-1,1)$ (see \cite[Proposition 2.7]{CL}),
\begin{equation*}
\lim_{t\rightarrow 0} \frac{E(u+tv)-E(u)}{t}=\bigl\langle -{\Delta}_{\gamma}u +u\, \mbox{Log}\left|u\right|^{2}-u,v\bigl\rangle
\end{equation*}
Thus, $E$ is G\'ateaux differentiable. Then, by Lemma \ref{APEX23} we see that $E$ is  Fr\'echet differentiable  and $E^{\prime}(u)=-{\Delta}_{\gamma}-u\,\mbox{Log}\left|u\right|^{2}-u$.
\end{proof}

\bibliographystyle{plain}
\bibliography{bibliografia}

\end{document}